\documentclass[12pt]{article}

\usepackage{amssymb,amsmath,amsthm,sectsty,url,mathrsfs,graphicx}
\usepackage[letterpaper,hmargin=1.0in,vmargin=1.0in]{geometry}
\usepackage[dvips,colorlinks,linkcolor=black,citecolor=black,filecolor=black,urlcolor=black]{hyperref}
\usepackage{color}

\sectionfont{\large}
\subsectionfont{\normalsize}
\numberwithin{equation}{section}

\newtheorem{theorem}{Theorem}[section]
\newtheorem{lemma}[theorem]{Lemma}
\newtheorem{proposition}[theorem]{Proposition}
\newtheorem{corollary}[theorem]{Corollary}

\renewcommand{\Pr}{ \mathrm P}

\newcommand{ \rel}{ t_{\mathrm{rel}} }
\newcommand{ \mix}{ t_{\mathrm{mix}} }

\newcommand{ \TV}{ \mathrm{TV} }

\newcommand{\eps}{\epsilon}

\newcommand{\la}{\lambda}

\DeclareMathSymbol{\leqslant}{\mathalpha}{AMSa}{"36} 
\DeclareMathSymbol{\geqslant}{\mathalpha}{AMSa}{"3E} 
\DeclareMathSymbol{\eset}{\mathalpha}{AMSb}{"3F}     

\renewcommand{\epsilon}{\varepsilon}

\newcommand{\N}{\mathbb N}
\newcommand{\R}{\mathbb R}
\newcommand{\Z}{\mathbb Z}

\usepackage[normalem]{ulem}

\begin{document}

\title{Relaxation times are stationary hitting times of large sets}
\author{Jonathan Hermon
\thanks{
Department of Mathematics, University of British Columbia, Canada. E-mail: {\tt jhermon@math.ubc.ca}. J.H.\ is supported by NSERC grants.}}
\date{}
\maketitle

\begin{abstract}
We give a characterization of the relaxation time up to an absolute constant factor, in terms of stationary expected hitting times of large sets. This resolves a conjecture of Aldous and Fill. We give a similar characterization for the spectral profile. We also provide in the non-reversible setup a related characterization for stationary expected hitting times of large sets.\end{abstract}


\section{Introduction}
In this work we give a characterization of the relaxation time up to an absolute constant factor, in terms of stationary hitting times of large sets. We denote the stationary distribution of an irreducible, reversible Markov chain $(X_m)_{m = 0}^{\infty}$ with transition matrix $P$ on a finite state space $V$ by $\pi$ and its \emph{relaxation time}, defined as the inverse of the spectral gap, by \[\rel:=(1-\text{the second largest eigenvalue of } P)^{-1}.\] For a set $A \subseteq V $ we denote its \emph{hitting time} by $T_A:=\inf\{m:X_m \in A \}$. We write $\pi_A$ for $\pi$ conditioned on $A$, i.e. \[\pi_A(a):=\frac{\pi(a)}{\pi(A)}\mathbf{1}\{a \in A\}.\] The following is our main result.
\begin{theorem}
\label{thm:1}
There exists an absolute constant $c_{1}>0$ such that for every irreducible reversible Markov chain on a finite state space $V$ we have that
\begin{equation}
\label{e:1}
c_{1}\rel \le t_{\mathrm{H}}^\pi:= \max_{A \subsetneq V  \, :\, \pi(A) \ge 1/2}\mathbb{E}_{\pi_{A^c}}[T_A] \le 2 \rel. 
\end{equation} 
\end{theorem}
Theorem \ref{thm:1}  resolves a conjecture of Aldous and Fill \cite[Open Problem 4.38]{aldous}, which asserts that under reversibility $\rel\le C \max_{A:A \subsetneq V }\pi(A) \mathbb{E}_{\pi_{A^c}}[T_A]$ for some absolute constant $C$ (note that $\max_{A:A \subsetneq V }\pi(A) \mathbb{E}_{\pi_{A^c}}[T_A] \le 2t_{\mathrm{H}}^\pi$). We shall concentrate on the discrete time setup.
As explained in \S\ref{s:ctstime}, all of the results from this paper concerning discrete-time Markov chains on a finite state space  are valid also in the continuous time setup (with an arbitrary Markov generator, not necessarily of the form $P-I$ for some transition matrix $P$).   This can easily be deduced  from the discrete time results.

Surprisingly, Theorem \ref{thm:1}   has an extension to the non-reversible setup. This is Theorem \ref{thm:3}, which is presented in \S\ref{s:nonrevintro}. In this extension we show that even without reversibility $t_{\mathrm{H}}^\pi$ is still comparable (up to an absolute constant multiplicative factor) to a certain generalization of the notation of the relaxation time to the non-reversible setup. Albeit the right notion is not the most straightforward one as it (crucially) involves time averaging.
 Namely,   we  show that  $t_{\mathrm{H}}^\pi$ is comparable to the minimal $t \ge 1 $ such that  the second largest eigenvalue of the multiplicative reversibilization of  $P^{G(t)}:=\sum_{i=0}^{\infty}\mathbb{P}[\eta_t=i] P^i=\sum_{i=0}^{\infty} (t-1)t^{-i}P^i$ (i.e., of $(P^{G(t)})^*P^{G(t)}$) is at most $1/e^{2}$, where $\eta_t +1 \sim \text{Geometric}(1/t)$. Let us recall that if $K$ is an irreducible transition matrix with a stationary distribution $\pi$, its \emph{time-reversal} $K^*$ is given by  $K^*(x,y):=\frac{\pi(y)}{\pi(x)}K(y,x)$ for all $x,y$. We say that $K$ is \emph{reversible} if $K=K^*$. 

In \S\ref{s:ncycle} we demonstrate the necessity of employing time averaging when extending Theorem \ref{thm:1} to the non-reversible setup. Our usage of time averaging in terms of the Geometric distribution is inspired by the elegant use of it in \cite{PS} by Peres and Sousi, who attribute the idea of employing such a time averaging to Oded Schramm.

The total variation distance between two distributions $\nu$ and $\mu$  on the same countable (possibly finite) set $V$ is defined as $\|\mu-\nu\|_{\TV}:=\frac 12\sum_{v \in V}|\mu(v)-\nu(v)|$. The $\eps$ total variation mixing time $t_{\mathrm{mix}}^{\mathrm{TV}}(\eps)$ is defined to as
\[t_{\mathrm{mix}}^{\mathrm{TV}}(\eps):=\max_{v \in V} \inf \{t \ge 0 : \| \mathbb{P}_v[X_t= \cdot]-\pi \|_{\TV} \le \eps \}. \]
 When $\eps=1/4$ we omit it from the notation and refer to    $t_{\mathrm{mix}}^{\mathrm{TV}}=t_{\mathrm{mix}}^{\mathrm{TV}}(1/4)$ as the total variation mixing time.

Under reversibility   $t_{\mathrm{mix}}^{\mathrm{TV}}$ admits a hitting time characterization in terms of a certain quantity $t_{\mathrm{H}}(1/2)$ defined in the following display. This is very similar to our characterization of  $\rel$ in terms of  $t_{\mathrm{H}}^\pi$. The difference is that in the definition of   $t_{\mathrm{H}}^\pi$, when considering the expectation of $T_A$, the initial distribution is taken to be  $\pi_{A^c}$, whereas  the definition of $ t_{\mathrm{H}}(1/2)$ involves the worst initial state. Namely, there exist  absolute constants $c,C>0$ such that for all irreducible, reversible Markov chains on a finite state space $V$ with a stationary distribution $\pi$, the total variation mixing time in the continuous time setup, and under the assumption that $\min_{x \in V}P(x,x) \ge \frac 12$ also in the discrete time setup, satisfies that
\[ct_{\mathrm{mix}}^{\mathrm{TV}}\le t_{\mathrm{H}}(1/2)\le C t_{\mathrm{mix}}^{\mathrm{TV}}, \quad \text{where} \quad  t_{\mathrm{H}}(\alpha):=\max_{A \subset V ,  \, x \in V :\, \pi(A) \ge \alpha}\mathbb{E}_{x}[T_A].   \]
 Aldous \cite{Aldoussome} first proved this with $\max_{A \subset V,  x \in V \, }\pi(A)\mathbb{E}_{x}[T_A] $ in place of $ t_{\mathrm{H}}(1/2)$. This was then refined independently by Oliveira \cite{Oliveira} and by Peres and Sousi \cite{PS}, who proved that for all $\alpha \in (0, 1/2)$ the above holds with  $ t_{\mathrm{H}}(\alpha)$ in place of $ t_{\mathrm{H}}(1/2)$, with constants $c=c_{\alpha}$ and $C=C_{\alpha}$ that depend on $\alpha$.  This was later strengthened to the assertion of the last display in \cite{half}, where it is shown that $t_{\mathrm{H}}(\alpha)\le t_{\mathrm{H}}(1/2)/\alpha$ for all $\alpha \in (0,1/2)$.

  The aforementioned results were refined by Basu, Peres and the author \cite{Basu}, who gave a characterization of the cutoff phenomenon  for reversible Markov chains in terms of concentration of hitting times of ``worst"large sets. In particular, it is shown in  \cite{Basu} that  there exists an absolute constant $C>0$ such that  for all $\eps \in (0,1)$ and all $0<\delta< \eps \wedge (1-\eps)$, where $a \wedge b:=\min \{a,b\}$ and $a \vee b:=\max \{a,b\}$, for all continuous-time reversible Markov chains on a finite state space $V$  we have that
\begin{equation}
\label{e:hitmix}
\rel \log \left(\frac{1}{2 \eps} \right) \vee \mathrm{hit}_{1-\delta}(\eps+\delta)  \le t_{\mathrm{mix}}^{\mathrm{TV}}(\eps) \le \mathrm{hit}_{1-\delta}(\eps-\delta)+C \log(1/\delta)\rel ,
\end{equation}
where for $\alpha,p \in (0,1)$, \[\mathrm{hit}_{\alpha}(p):=\inf \{t:\mathbb{P}_x[T_A>t] \le 1-p \text{ for all }x \in V \text{ and all }A \subset V \text{ such that }\pi(A) \ge \alpha \}.\]
In the discrete time setup the same holds with $\rel$ replaced by $t_{\mathrm{rel}}^*-1$, where $t_{\mathrm{rel}}^*$ is the absolute relaxation time, defined as the maximum of $\rel$ and the inverse of $1+\lambda_{|V|}$, where $\lambda_{|V|}$ is the smallest eigenvalue of $P$ (see  \cite[Remark 1.9]{Basu}).
Another similar two-sided inequality is established in \cite{Basu} in which $\mathrm{hit}_{1-\delta}$ above is replaced with $\mathrm{hit}_{1/2}$.
See \cite{tech} for some related results, as well as some results related to \cite{half}, such as the bound \cite[Proposition 1.13]{tech}\[t_{\mathrm{H}}(1-\eps)-t_{\mathrm{H}}(\eps) \le C \eps^{-1} \log (1/\eps) \rel \quad \text{for all } \eps \in (0,1/2). \] 
In \cite{HP} the author and Peres obtained a characterization (up to a universal constant) of the $\ell^{\infty}$ mixing time (see \eqref{e:tmixinftydef} for a definition) of reversible Markov chains in terms of tails of hitting times.
Namely, under reversibility, the $\ell^{\infty}$ mixing time is comparable to the minimal $t$ such that  $\mathbb{P}_x[T_{B^c}>t] \le \pi(B) $ for all initial states $x$ and all $B \subset V$ such that $\pi(B) \le 1/2$.
A similar characterization is given in \cite{HP} for the relative entropy mixing time.   
  
As we now explain, Aldous and Fill's conjecture has a natural interpretation in terms of a comparison of the expected exit time from a set $B$ starting from initial distribution $\pi_B$ (the stationary distribution conditioned on $B$, i.e., $\pi_B(b):=\frac{\mathbf{1}\{b \in B\} \pi(b)}{\pi(B)}$) with  the expected exit time from $B$ starting from the \emph{quasi-stationary distribution of} $B$, denoted by $\alpha^B$.\footnote{We write $\alpha^B$ instead of $\alpha_B$ to avoid the interpretation ``$\alpha$ conditioned on $B$".  In the proof of Theorem \ref{thm:2} we shall suppress the dependence on $B$, and then for $D \subset B$ write $\alpha_D$ for $\alpha^B$ conditioned on $D$.} As we later explain, the latter expected exit time is intimately related to the spectral profile.

Theorem \ref{thm:1} follows from our second theorem, Theorem \ref{thm:2}, which asserts that for some absolute constant $c>0$, for all irreducible, reversible Markov chains with a finite state space $V$ and stationary distribution $\pi$, we have  for all $\eset \neq B \subsetneq V $ that 
\begin{equation}
\label{e:piqsintro}
c\mathbb{E}_{\alpha^{B}}[T_{B^c}]\le \max_{D \,: \, D \subseteq B}\mathbb{E}_{\pi_{D}}[T_{D^c}] \le\mathbb{E}_{\alpha^{B}}[T_{B^c}].
\end{equation}
Consider a reversible Markov chain with a finite state space $V$. Denote the transition matrix of the chain by $P$.  Let $\eset \neq B \subsetneq V$. 
Denote the transition matrix of the chain killed upon hitting $B^{c}$  by  $P_B$. In other words, this is the \emph{restriction} of the matrix $P$ to $B$, obtained by deleting from $P$ the rows and columns corresponding to elements of $B^c$, resulting in the substochastic square matrix indexed by $B$, satisfying that $P_B(x,y)=P(x,y)$ for all $x,y \in B$. By abuse of notation, we shall also   use $P_B$ to  refer  to the substochastic matrix labeled by $V$, obtained by setting the entries in the rows and columns of $P$ corresponding to $B^c$ to zero. That is,
 $P_B(x,y)=P(x,y)\mathbf{1}\{x,y \in B\}$.
Denote the largest eigenvalue of $P_{B}$ by $\beta(B):= 1-\lambda(B)$. When we want to emphasize the underlying transition matrix we write  $\beta_{P}(B):= 1-\lambda_{P}(B)$.  By the Perron--Frobenius theorem there exists a distribution $\alpha^B $ on $B$ satisfying that $\alpha^B P_B=\beta(B)\alpha^B$, from which it is not hard to deduce that $\mathbb{P}_{\alpha^{B}}[T_{B^{c}}>k]=\beta(B)^{k}$ for all $k$ and so $\mathbb{E}_{\alpha^{B}}[T_{B^{c}}]=1/\lambda(B)$. The distribution $\alpha^B$ is called the \emph{quasi-stationary distribution} of $B$. The Perron--Frobenius theorem also asserts that if $P_B$ is irreducible, meaning that $\sum_{n=0}^{\infty}(P_B)^n(x,y)>0$ (equivalently, $\mathbb{P}_x[T_y<T_{B^c}]>0$) for all $x,y \in B$, then $\alpha^B$ is the unique distribution which is a left eigenvector of $P_B$. If $P_B$ is reducible, we pick $\alpha^B$ to be a distribution which is a left eigenvector of $P_B$ with the largest eigenvalue, which we denote by $\beta(B)=1-\lambda(B)$, among all  distribution which are left eigenvectors of $P_B$.  

As we now explain, it is well-known that \eqref{e:1} indeed holds if one replaces the quantity $t_{\mathrm{H}}^\pi=\max_{A \, :\, \pi(A) \ge 1/2}\mathbb{E}_{\pi_{A^c}}[T_A] $ with  $\max_{A \, :\, \pi(A) \ge 1/2}\mathbb{E}_{\alpha^{A^{c}}}[T_A] $. Indeed, in \S\ref{s:background} we prove the standard fact that for all reversible Markov chains on a finite state space\footnote{We require $|V| \ge 2$ as part of our definition of a Markov chain.} $V$ there exists $\eset \neq B_{1} \subsetneq V$ with $\pi(B_{1}) \le 1/2$ whose quasi-stationary distribution $\alpha^{B_{1}}$ satisfies that
\begin{equation}
\label{e:almost}
\rel \le 1/\lambda(B_{1})=\mathbb{E}_{\alpha^{B_{1}}}[T_{B_1^c}]\le 2 \rel,
\end{equation}
and the fact that for all $\eset \neq B \subsetneq V$ it holds that
\begin{equation}
\label{e:almost2}
\mathbb{E}_{\pi_{B}}[T_{B^{c}}] \le\mathbb{E}_{\alpha^{B}}[T_{B^c}]= 1/\lambda(B). \end{equation}
In the seminal work  \cite{AB} Aldous and Brown proved that for reversible Markov chains, for all $B \subset V$ and $t$
\[ 1-\frac{\rel}{\mathbb{E}_{\alpha^{B}}[T_{B^c}]}  \le \frac{\mathbb{P}_{\pi}[T_{B^c}>t]}{\beta(B)^{t}} =\frac{\mathbb{P}_{\pi}[T_{B^c}>t]}{\mathbb{P}_{\alpha^{B}}[T_{B^c}>t]}\le  1 \quad \text{and so }\]
 \[\mathbb{E}_{\alpha^{A^{c}}}[T_A]-\rel \le  \mathbb{E}_{\pi}[T_{A}] \le \mathbb{E}_{\alpha^{A^{c}}}[T_A] \quad \text{for all } A \subsetneq V.\footnote{For a recent refinement of this inequality see \cite[Theorems 5.2, 5.4 and 5.5]{BHT}.}  
\]
In conjunction with \eqref{e:almost} and \eqref{e:almost2}, at first glance this appears to come tantalizingly close to proving 
 $\rel\lesssim \max_{A}\pi(A) \mathbb{E}_{\pi_{A^c}}[T_A]$. Indeed, fix a small absolute constant $c \in (0,1/2)$. Consider a set $A$ which minimizes  $\la(A^{c})$ among all  $ A \subset B_1^c$ such that $\pi(A) \ge c$, where $B_1$ is as in \eqref{e:almost}. While it is always the case that $\lambda(B') \le \lambda(B)$ whenever $B\subseteq B'$ (see \eqref{e:lambda(A)CF}), one might hope that if $c$ is sufficiently small, then for this  $A$ (which satisfies $A ^{c}\supset B_1$)  instead of merely having $\lambda(A^{c}) \le \lambda(B_{1}) $ it holds that $(1+c)\lambda(A^{c})\le \lambda(B_{1}) $. By \eqref{e:almost} this would imply that $ \mathbb{E}_{\alpha^{A^{c}}}[T_A]\ge (1+c)\rel$. Plugging this into the Aldous--Brown inequality would then yield that $\mathbb{E}_{\pi_{A^c}}[T_A] \ge\mathbb{E}_{\pi}[T_A] \ge c \rel$, as desired. 

 Alas,
the example which we study in \S\ref{s:example} shows that   for simple random walk on an $n$-vertex graph of maximal degree 4 it is possible that $ \rel \asymp n^{2/3}$ while for all fixed $\delta \in (0,\frac 12]$, $\max_{B:\pi(B) \in [\delta ,1-\delta]}\mathbb{E}_{\pi_{B}}[T_{B^c}] \le C(\delta)$ for some constant $C(\delta)$ which  depends only on $\delta$.  It follows that one cannot replace $\max_{D \,: \, D \subseteq B}\mathbb{E}_{\pi_{D}}[T_{D^c}] $ in \eqref{e:piqsintro}  with $\mathbb{E}_{\pi_{B}}[T_{B^c}]$ and that for some sequences of reversible Markov chains the maximum in the definition $t_{\mathrm{H}}^\pi $ from \eqref{e:1} can only be attained  by a sequence of sets $A$ with $\pi(A^c)=o(1)$ (suppressing the dependence of $A$ and of $\pi$ on the index of the chain in the sequence).

Let $\pi_*:=\min_x \pi(x)$. Introduced by Goel et al.\ \cite{spectral}, the \emph{spectral profile }  $\Lambda=\Lambda_P:[\pi_*,\infty) \to (0,1] $ of $P$ is given by $\Lambda(\delta)=\Lambda(1-\pi_*)$ for all $\delta > 1-\pi_*$ and for $\delta \in (0,1-\pi_*]$, 
\begin{equation}
\label{e:spdefintro1}
\begin{split}
\Lambda(\delta): & =\min \{\lambda (B): B \subset V,\, 0<\pi(B) \le \delta\}
\\ & =\min \left\{\frac{1}{\mathbb{E}_{\alpha^B}[T_{B^c}]} : B \subset V,\,0<\pi(B) \le \delta \right\}.
\end{split}
\end{equation}
We note that $\Lambda(\delta)=\min \{\lambda (B): B \subset V,\,0<\pi(B) \le \delta,\, B \text{ is irreducible}\}$.
Using
\eqref{e:piqsintro} we can compare $\Lambda(\delta)$ with the quantity $\kappa(\delta)$, defined via \[\kappa(\delta):=\min \left\{\frac{1}{\mathbb{E}_{\pi_{B}}[T_{B^{c}}]} :B \subset V,\, 0<\pi(B) \le \delta \wedge (1-\pi_*), \, B \text{ is irreducible}  \right\}.  \]
\begin{theorem}
\label{thm:2}
There exists an absolute constant $c_{2}>0$ such that for all irreducible, reversible Markov chains on a finite state $V$,  for all $B \subsetneq V $ we have that 
\begin{equation}
\label{e:2}
\frac{c_{2}}{\lambda(B)} \le \max_{D \,: \, D \subseteq B}\mathbb{E}_{\pi_{D}}[T_{D^{c}}] \le \frac{1}{\lambda(B)}.
\end{equation}
Hence for all $\delta \in [\pi_*,1]$
\begin{equation}
\label{e:3}
\Lambda(\delta) \le \kappa(\delta)\le \frac{1}{c_{2}} \Lambda(\delta). 
\end{equation}
\end{theorem}
We present an extension of Theorem \ref{thm:2} to the case of reversible Markov chains on a countably infinite state space in \S\ref{s:countableintro}. In particular, in \S\ref{s:countableintro}  we allow  $\pi(V)=\infty$. In \S\ref{s:SP} we present a surprising conceptual consequence of Theorem \ref{thm:2} (more precisely, of the proof of Theorem \ref{thm:2}).
\subsection{The spectral profile}
 The spectral profile was first introduced in the context of the quantitative study of Markov chains on finite state spaces in  \cite{spectral}. The authors of  \cite{spectral} defined it in a slightly different fashion. We denote the function they call the spectral profile by $\widehat \Lambda$. We shall recall its definition below (see \eqref{e:spwidehatdefintro}). In \S\ref{s:1.9} we reproduce  the argument from  \cite{spectral} which shows that  
\begin{equation}
\label{e:spectralprofiletwodefinitionsareequiv}
\Lambda(\delta) \le \widehat \Lambda(\delta) \le (1-\delta)^{-1} \Lambda(\delta) \quad \text{for all }\delta \in [\pi_*,1).
\end{equation}
In particular, for $\delta \le 1/2$ the ratio of the two quantities is  at most  2. 

We now recall two general theoretical applications of the spectral profile. A fundamental quantity associated with a Markov chain is its \emph{log-Sobolev constant} $c_{\mathrm{LS}}$. In their seminal work \cite{LS}, Diaconis and Saloff-Coste established  intimate quantitative ties between  $c_{\mathrm{LS}}$   and the $\ell^{\infty}$ mixing time (in the reversible setup) as well as with the remarkable general phenomenon of hypercontractivity. Moreover, they showed that  $c_{\mathrm{LS}}$  is amenable to the method of comparison of Dirichlet forms (via the canonical paths method), allowing one to obtain an estimate of   $c_{\mathrm{LS}}$  for one Markov chain in terms of that of another (and the maximal congestion of an edge; see e.g.\ \cite[\S13.4]{levin} for a gentle introduction to this method). The spectral profile is also amenable to the method of comparison of Dirichlet forms.    It is shown in \cite{HP} that  for all irreducible Markov chains on a finite state space
\[\frac{c_{\mathrm{LS}}}{17} \le \min_{\delta \in [\pi_{*} ,1/2]} \frac{\Lambda(\delta)}{\log (1/\delta)} \le  c_{\mathrm{LS}}.  \]
\begin{corollary}
\label{cor:LS}
Let $c_2$ be the absolute constant from \eqref{e:3}. Then for all irreducible reversible Markov chains on a finite state space
\[\frac{c_{\mathrm{LS}}}{17} \le \min_{\delta \in [\pi_{*} ,1/2]} \frac{ \kappa(\delta)}{\log (1/\delta)} \le  \frac{c_{\mathrm{LS}}}{c_2}. \]
\end{corollary}
\noindent We now describe the main results from \cite{spectral}. These results do not assume reversibility, and we shall describe them in their full generality. It is proved in \cite{spectral} for $\eps \in (0,1]$ that the $\varepsilon$ $\ell^{\infty}$ mixing time
\begin{equation}
\label{e:tmixinftydef}
t_{\mathrm{mix}}^{(\infty)}(\eps):=\min \left\{t: \max_{x,y} \left|\frac{P_t(x,y)-\pi(y)}{\pi(y)} \right| \le \eps \right\}
\end{equation}
 of the continuous-time Markov chain  whose time-$t$ transition probabilities are given by the matrix $P_t:=e^{t(P-I)}$  (where $P$ is not assumed to be reversible) satisfies that 
\begin{equation}
\label{e:goeletal}
 t_{\mathrm{mix}}^{(\infty)} (\eps)\le 8 \log (e/\varepsilon) \rel \left(\frac{P+P^{*}}{2} \right)+8\int_{1/\pi_*}^{1/2}\frac{\mathrm{d}\delta}{\delta \widehat\Lambda_{P}(\delta)},
\end{equation}
where throughout, when we include a transition matrix $Q$ in some quantity like $\rel(Q)$ or $\widehat \Lambda_{Q}$, we do so to indicate that the quantity is taken w.r.t.\ $Q$.
In discrete time (i.e., when $P_t(x,y)$ in \eqref{e:tmixinftydef} is replaced with $P^t(x,y)$) they prove that $t_{\mathrm{mix}}^{(\infty)}(\eps) \le 16\log (e/\varepsilon)\rel(PP^{*})+16\int_{1/\pi_*}^{1/2}\frac{\mathrm{d}\delta}{\delta \widehat \Lambda_{PP^{*}}(\delta)}$. Moreover, they proved that  $\min_{x}P(x,x) \rel(PP^{*}) \le   \rel \left(\frac{P+P^{*}}{2} \right)$ and that
\begin{equation}
\label{e:lazinessinsp}
\text{for all } \delta, \quad \widehat \Lambda_{PP^{*}}(\delta) \ge  \widehat  \Lambda_{P}(\delta) \min_{x}P(x,x).
\end{equation}
Kozma \cite{Kozma} showed that for reversible chains \eqref{e:goeletal}  is sharp up to a $O(\log \log (1/\pi_* ))$ factor. 

Before giving the definition of $\widehat \Lambda$ we require some additional definitions and notation. Let $f,g,h \in \mathbb{R}^V$ and $p \in [1,\infty)$. We define the inner product $\langle \cdot,\cdot \rangle_{\pi}$ on $\mathbb{R}^V$ associated  with $\pi$, the  stationary expectation and variance and the $\ell^p$ norm of $f $ to be  $\|f\|_p:=(\mathbb{E}_{\pi}[|f|^{p}])^{1/p}$,  \[\mathbb{E}_{\pi}[g]:=\sum_{x}\pi(x)g(x), \quad   \langle g,h \rangle_{\pi}:=\mathbb{E}_{\pi}[gh] \quad \text{and} \quad \mathrm{Var}_{\pi}f:=\mathbb{E}_{\pi}[f^{2}]-(\mathbb{E}_{\pi}[f])^{2}.
\]
 The $\ell^{\infty}$ norm of $f$ is defined as  $\|f\|_{\infty}:=\max_{v \in V}|f(v)|$. The $\ell^p$ norm $\|\cdot\|_{p,\pi}$  on the space of signed-measures on $V$ is defined via the relation $\|\sigma\|_{p,\pi}=\|\frac{\sigma}{\pi }\|_p$, where $\frac{\sigma}{\pi }\in \mathbb{R}^V$ is given by $\frac{\sigma}{\pi }(x):=\frac{\sigma(x)}{\pi (x) }$. Let  $\mathcal{P}(V)$ be the collection of probability measures on $V$. 

Denote the support of $f \in \mathbb{R}^V $ by $\mathrm{supp}(f):=\{x \in V :f(x) \neq 0\}$. Let  \[ C_0(B):=\{f \in \mathbb{R}^V:\mathrm{supp}(f) \subseteq B \} \quad \text{and} \quad C_0^{+}(B):=\{f \in C_0(B):f \ge 0  \}.\] 
The spectral profile corresponding to $P$ is  defined in  \cite{spectral}  to be
\begin{equation}
\label{e:spwidehatdefintro}
\widehat \Lambda(\delta)=\widehat \Lambda_P(\delta) :=\inf\{\widehat \lambda(B): \pi(B) \le \delta \}, \quad \text{where}
\end{equation}
\begin{equation}
\label{e:widehatlambdaBdef}
\widehat \lambda(B)=\widehat \lambda_{P}(B):=\min\{\langle (I-P)f,f \rangle_{\pi}/\mathrm{Var}_{\pi}f : \mathrm{Var}_{\pi}f  \neq 0,f \in C_0(B) \}. 
\end{equation}
Note that for all $\delta \ge 1$ we have that $\widehat \Lambda(\delta)=\widehat \Lambda (1)$ is the spectral gap of $\frac{P+P^{*}}{2}$.

\subsection{The non-reversible setup}
\label{s:nonrevintro}
We now drop the assumption of reversibility.\footnote{While we presented some background on the spectral profile without assuming reversibility, Theorems \ref{thm:1} and \ref{thm:2} and Corollary \ref{cor:LS} assumed $P$ is reversible. The next theorem will not include this assumption. } We shall only assume that $P$ is irreducible and as always denote its stationary distribution by $\pi$ and its finite state space by $V$. Recall that
the \emph{time-reversal} of $P$, denoted by  $P^{*}$, is   given by the relation $\pi (x)P^{*}(x,y)=\pi(y)P(y,x)$ for all $x,y \in V$. It is uniquely defined also via the relation $\langle Pf,g \rangle_{\pi}=\langle f,P^{*}g \rangle_{\pi}$ for all $f,g \in \mathbb{\mathbb{R}}^V$. It is  easy to see that since $\pi$ is the stationary distribution of $P$, and $P$ is stochastic and irreducible  (in particular, $\pi(x)>0$ for all $x$), then $P^{*}$ is also an  irreducible, stochastic matrix, and that $\pi$ is its  stationary distribution. Moreover, $(P^*)^*=P$ and $(P^k)^*=(P^*)^k$ for all $k \in \mathbb{N}$.   

Recall that for $\eset \neq  B \subsetneq V$ we denote the restriction of $P$ to $B$ by $P_B$.  We can similarly define $(P_B)^*$ via the relation $\pi (x)(P_{B})^{*}(x,y)=\pi(y)P_{B}(y,x)$ for all $x,y \in V$. Clearly,   $(P_B)^*=(P^{*})_B$ and hence we shall simply write $P_B^*$.

Let   $\mathbf{X}=(X_k)_{k=0}^{\infty}$ be a realization of the Markov chain with transition matrix $P$.
Let $t \in [1,\infty) $. Let $\eta_t +1\sim \mathrm{Geometric}(1/t)$ (i.e.\  $\mathbb{P}(\eta_t=i)=t^{-1}(1-1/t)^{i}$ for all $i \in \Z_+$)   be independent of    $\mathbf{X}$. We define
\begin{equation}
\label{e:defofQgeomp20}
\begin{split}
P^{\mathrm{G}(t  )}(x,y) & :=\mathbb{P}_x[X_{\eta_t   }=y]=\sum_{m \in \Z_+}P^m(x,y)\mathbb{P}[\eta_t=m] =\frac{1}{t} \sum_{m \in \mathbb{Z}_+}P^m(x,y)\mathbb{P}[\eta_t\ge m] .
 \end{split}
\end{equation}

Let $1^{\perp}:=\{f \in \mathbb{\mathbb{R}}^V:\mathbb{E}_{\pi}[f]=0 \}$. Let $Q$ be an irreducible transition matrix whose stationary distribution is $\pi$. It is easy to see that $1^{\perp}$ is a $Q$ invariant subspace of $\mathbb{R}^V$. Let $\Pi$ be the transition matrix that all of its rows are equal to $\pi$. The operator norm of $Q-\Pi$  is defined to be
\begin{equation}
\label{e:introdefofoperatornorm}
\begin{split}
\|Q-\Pi\| & :=\max \{\|(  Q-\Pi )f\|_{2}  : f \in \mathbb{\mathbb{R}}^V, \|f \|_{2}=1  \}
\\ & = \max \{\|  Qf\|_{2}  : f \in  1^{\perp}, \|f \|_{2}=1  \}=( \max \{  \mathrm{Var}_{\pi}Qf  :   \mathrm{Var}_{\pi}f=1    \})^{1/2}. 
\end{split}
\end{equation}
It follows from the second equality in the last display that for all $k \in \mathbb{N}$,
\begin{equation}
\label{e:submulton}
\|Q^k-\Pi\| \le \|Q-\Pi\|^k.
\end{equation}

The \emph{multiplicative reversibilization} of a stochastic  matrix $Q$ on $V$ whose stationary distribution is $\pi$ is defined to be $S_Q:=Q^*Q$.  Note that $(S_Q)^*=S_Q$.  Since $ Q$ and $ Q^*$ are stochastic and have $\pi$ as their stationary distribution, we have that $\Pi Q=\Pi=\Pi Q^*$,  $Q\Pi =\Pi=Q^*\Pi $. This, together with   $\Pi=\Pi^{*}$, implies that $Q^*Q-\Pi=S_{Q-\Pi}$ and  $QQ^*-\Pi=S_{Q^{*}-\Pi}$.  It is standard that 
\begin{equation}
\label{e:introdefofoperatornorm2}
 \|Q-\Pi \|^2 =\|S_{Q-\Pi}\|= \beta_{2}(Q^*Q)=\beta_{2}(QQ^*)=\|S_{Q^{*}-\Pi}\|=\|Q^{*} -\Pi\|^2,
\end{equation}
where for a reversible transition matrix $P$ we denotes the $i$th largest eigenvalue of $P$ by   $\beta_{i}(P)=1-\lambda_i(P)$ and its relaxation time by $ \rel(P)=\frac{1}{1-\beta_2(P)}$.
 
We shall prove a version of Theorem \ref{thm:1} with the following quantity replacing the relaxation time of $P$ (recall that $S_{P^{\mathrm{G}( t)}}=(P^{\mathrm{G}( t)})^{*}P^{\mathrm{G}( t)}=(P^*)^{\mathrm{G}( t)}P^{\mathrm{G}( t)}$):
\begin{equation}
\label{e:introdefotrelgeom}
\rel^{\text{Geom}}(\delta) :=\inf \{t \ge1 :\beta_2(S_{P^{\mathrm{G}( t)}}) \le \delta^2 \}=\inf \{t \ge 1:\|P^{\mathrm{G}( t)}-\Pi\| \le \delta \}.
\end{equation}
It follows from Lemma \ref{lem:Geomsubmultiplicativity} that the map $t \mapsto \|P^{G(t)} -\Pi\|  $ is continuous and non-increasing\footnote{In fact, with additional effort one can even deduce from Lemma \ref{lem:Geomsubmultiplicativity} that the map $t \mapsto \|P^{G(t)} -\Pi\|  $  is strictly decreasing in $t$, but we shall not require this.} in $t$, and that for all $\eps \in (0,1/2]$,
\begin{equation}
\label{e:introdefotrelgeom'}
\frac{\eps}{1-\eps} \rel^{\mathrm{Geom}} (\eps) \le \rel^{\mathrm{Geom}} (1/2) \le \frac{1-\eps}{\eps} \rel^{\mathrm{Geom}} (1-\eps).  
\end{equation}
In particular, the choice of $\eps=1/e$ in Theorem \ref{thm:3} is arbitrary.

Recall that $ t_{\mathrm{H}}^\pi= \max_{A \, :\, \pi(A) \ge 1/2}\mathbb{E}_{\pi_{A^c}}[T_A] $.

\begin{theorem}
\label{thm:3}
There exists an  absolute constant $c_{3}>0$ such that for every irreducible  Markov chain on a finite state space we have that
\begin{equation}
\label{e:3T3}
c_{3}\rel^{\mathrm{Geom}} (1/e)\le t_{\mathrm{H}}^\pi\le \frac{1}{1-\sqrt{b} } \rel^{\mathrm{Geom}}(1/e), \quad \text{ where} \quad b:=\frac{1+e^{-2}}{2} . 
\end{equation} 
\end{theorem}

\subsubsection{The necessity of time averaging in Theorem \ref{thm:3}}
\label{s:ncycle}
We now discuss the necessity of the time averaging in our definition of $\rel^{\mathrm{Geom}} (1/e)$. Before discussing a natural quantity, which is similar to  $\rel^{\mathrm{Geom}} (1/e)$ but does not involve time averaging, we first recall that  the $\eps$ $\ell^{2}$ mixing time is defined as     \[t_{\mathrm{mix}}^{(2)}(\eps):=\min \{t: \max_x \|\mathbb{P}_x[X_t = \cdot] -\pi\|_{2,\pi} \le \eps \}.\] It is proved in \cite{spectral}  that $t_{\mathrm{mix}}^{(\infty)}(\eps) \le 2t_{\mathrm{mix}}^{(2)}(\sqrt{\varepsilon})$ (under reversibility this is an equality in the continuous time setup \cite[Proposition 4.15]{levin}). In particular $t_{\mathrm{mix}}^{(\infty)}(1/4) \le 2 t_{\mathrm{mix}}^{(2)}(1/2) $.  

Another extension of the notion of the relaxation time to the non-reversible setup, which is in fact more natural from the perspective of its relation with the total variation and $\ell^{\infty}$ mixing times of the chain, is the $\eps$ \emph{pseudo relaxation time}
\[t_{\mathrm{rel}}^{\mathrm{ps}}(\eps):=t_{\mathrm{rel}}^{\mathrm{ps}}(P,\eps):=\inf \{k:\beta_2(S_{P^{k}}) \le \eps^2 \}=\inf \{k:\|P^{k}-\Pi\| \le \eps \}. \]
By \eqref{e:submulton} we have that $t_{\mathrm{rel}}^{\mathrm{ps}}(\eps^{2}) \le 2 t_{\mathrm{rel}}^{\mathrm{ps}}(\eps)$ for all $\eps$. Using this, one can deduce that (see \cite{Paulin}) 
\[t_{\mathrm{mix}}^{(\infty)}(1/4) \le 2 t_{\mathrm{mix}}^{(2)}(1/2) \lesssim t_{\mathrm{rel}}^{\mathrm{ps}}(1/e)  \log(1/\pi_*)  \quad \text{and that} \quad t_{\mathrm{rel}}^{\mathrm{ps}}(1/e) \asymp 1/\gamma_{\mathrm{ps}},  \]
where $\gamma_{\mathrm{ps}}:=\max \left\{\frac{1-\beta_2(S_{P^{k}}) }{k}:k \ge 1 \right\}$ is the \emph{pseudo spectral gap}, introduced by Paulin \cite{Paulin}. In  \cite{Paulin} Paulin demonstrated that in several senses $\gamma_{\mathrm{ps}}$ serves as a good extension in the non-reversible setup of the notion of the absolute spectral gap. One of them is \cite[Prop 3.3]{Paulin} 
 \[\left(\gamma_{\mathrm{ps}}^{-1}-1 \right) \log 2\le t_{\mathrm{mix}}^{\TV} \quad \text{for all }\eps \in (0,1). \] 
It is fairly straightforward to verify using \eqref{e:introdefotrelgeom'} and  $t_{\mathrm{rel}}^{\mathrm{ps}}(\eps^{2}) \le 2 t_{\mathrm{rel}}^{\mathrm{ps}}(\eps)$  that \[\rel^{\mathrm{Geom}} (1/e) \lesssim  t_{\mathrm{rel}}^{\mathrm{ps}}(1/e) \asymp \gamma_{\mathrm{ps}}^{-1}.\]

Consider a biased random walk on the $n$-cycle, whose transition matrix is given by $P(i,i+1)=1/3=2P(i,i-1)$ (where $i \pm 1$ are defined mod $n$) and $P(i,i)=\frac 12$ for all $i$. It is non-reversible. Since $P(i,i) =\frac 12$ for all $i$ and $PP^*=P^*P$ we have that \[\gamma_{\mathrm{ps}} \asymp 1-\beta_2(P^*P) \asymp 1-\beta_2 \left(\frac 12(P^*+P) \right) \asymp n^{-2} \asymp \left( t_{\mathrm{H}}^{\pi}  \right)^{-2}.  \] 
It follows that we cannot replace $\rel^{\mathrm{Geom}} (1/e)$ in \eqref{e:3T3} with $1/\gamma_{\mathrm{ps}}$ nor with $ t_{\mathrm{rel}}^{\mathrm{ps}}(1/e) $.
\subsection{Countable state space}
\label{s:countableintro}
In this subsection we  present an extension of Theorem \ref{thm:2} to the setup that $V$ is countably infinite  and $P$ is irreducible and is reversible w.r.t.\ $\pi$. We allow for the case that there exists $B \subset V$ such that $\pi(B)<\infty$ and $|B|=\infty$. In particular  $\pi_*:=\inf_{v \in V}\pi(v)$ may equal zero. 
In this section we need to modify the meaning of some previously introduced notation. To avoid  clash of notation, we shall prove all results concerning the case that $V$ is infinite  in the last section of the paper. 

  Let us first note that $\pi_D$ is still well-defined by the expression $\pi_D(x)=\frac{\pi(x)\mathbf{1}\{x \in D\}}{\pi(D)}$ as long as $\pi(D)<\infty$. We also note that the definitions of the $\ell^p$ norm of $f:V \to \mathbb{R}$ for $p \in [1,\infty)$ and of the inner product $\langle \cdot,\cdot \rangle_{\pi}$ are still valid if we interpret $\mathbb{E}_{\pi}[g]$ as $\sum_{v \in V}\pi(v) g(v)$. We define $\|f\|_{\infty}:=\sup_{v \in V}|f(v)|$.   
 We write $\ell^p(V,\pi)$ for the collection of all  $f:V \to \mathbb{R}$ satisfying that $\|f\|_p<\infty$.   

As before, for $B \subset V$ let $P_B$ be the restriction of $P$ to $B$. As before, it is convenient to  think of $P_B$ as being labeled by $V$  by setting $P_B(x,y):=P(x,y)\mathbf{1}\{x,y \in B\}$ for all $x,y \in V$.  The operator norm of $P_B$ is given by $\|P_B\|:=\sup \{\|P_B f\|_2 : f \in  \ell^2(V,\pi), \|f\|_2=1 \}$. 

We now modify the definition of $\Lambda(\cdot)$. Recall that $ \inf \eset=\infty$.
 For $u>0 $ we define  
 \[\Lambda(u):=\inf \{1-\|P_B\|  :B \subset V,  0<\pi(B) \le u \}.\] We extend the definition of $\kappa(u)$ for $u>0 $ to the current setup by setting:  \[\kappa(u):=\inf \left\{1/\mathbb{E}_{\pi_{B}}[T_{B^{c}}] :B \subset V,\, 0<\pi(B) \le u  \right\}.  \]
Recall that for a finite $D \subset V$ we write $\beta(D)=1-\la(D)$ for the largest eigenvalue of $P_D$. For $B \subseteq V$, let  $\mathrm{FI}(B)$ be the collection of all non-empty, finite, irreducible subsets of $B$, and for $u > 0$ let \[\mathrm{FI}(B,u):=\{D \in  \mathrm{FI}(B): \pi(D) \le u  \}.\]
 Recall that for any finite $D$ we have that $\|P_D\|=\beta(D)=1-\la(D)$, where $\beta(D)$ is the largest eigenvalue of $P_D$.
 \begin{theorem}
\label{thm:LambdaPucountable} Consider the above setup. Let $B \subset V$ be such that $0<\pi(B)<\infty$. Then
\begin{equation}
\label{e:countable1}
\|P_B\|  =\sup \{\beta(D):D \in \mathrm{FI}(B) \}.
\end{equation}
Consequently, 
\begin{equation}
\label{e:countable2}
\Lambda(u)=\inf \{\lambda(D): D \in \mathrm{FI}(V,u)  \}, \; \text{for all }u.
\end{equation}
Moreover, if in addition $B$ is also irreducible, then we  have that
\begin{equation}
\label{e:countable3}
\mathbb{E}_{\pi_{B}}T_{B^c}  =\sup \{\mathbb{E}_{\pi_{D}}T_{D^c}  : D \in \mathrm{FI}(B) \}.
\end{equation}
Consequently,
there exists an absolute constant $c>0$ (independent of $P$ and $B$) such that
\begin{equation}
\label{e:countable4}
c \left(1-\|P_B\|  \right)^{-1} \le \sup_{D \in \mathrm{FI}(B) }\mathbb{E}_{\pi_{D}}[T_{D^c}] \le \left(1-\|P_B\|  \right)^{-1}.
\end{equation} 
\begin{equation}
\label{e:countable5}
\kappa(u)=\inf \left\{\frac{1}{\mathbb{E}_{\pi_{D}}T_{D^c}}  : D \in \mathrm{FI}(V,u)  \right\}, \quad \text{for all }  u, \quad \text{and}
\end{equation}
\begin{equation}
\label{e:countable6}
c \kappa(u) \le \Lambda(u)\le  \kappa(u),  \quad \text{for all }u.
\end{equation} 
\end{theorem}    
\subsection{A sharp variant of Cheeger's inequality}
\label{s:SP}
Let us focus on the case that $P$ is reversible and irreducible, $V$ is countably infinite and $\pi(V)=\infty$. We now present a certain surprising conceptual consequence of our results.  As we explain below, it  can be seen as a certain sharp variant of Cheeger's inequality. Cheeger's inequality is also sharp. The below variant is sharper.  A similar result holds when $\pi(V)<\infty$. However, we shall state and prove the result only in the case that $\pi(V)=\infty$.    

We note that $\|1_B\|_2^2=\pi(B)$. The \emph{isoperimetric profile} $\phi:(0,\infty) \to [0,1] \cup \{ \infty\} $ is defined as (note that $\phi(u)=\infty$ if and only if $\pi(v)>u$ for all $v $;  the same holds for $\Lambda(u)$)
\[\phi(u):= \inf \left\{\frac{\langle( I-P)1_{B} ,1_{B} \rangle_{\pi}}{\|1_B\|_2^2} : B \in \mathrm{FI}(V, u)    \right\}=\inf \{\mathbb{P}_{\pi_{B}}[X_0 \in B,X_1 \notin B] : B \in \mathrm{FI}(V, u)\}, \]
where in the rightmost term above we consider the discrete time chain. The \emph{Cheeger constant} of $P$ is defined as $\Phi_*=\lim_{u \to \infty}\phi(u)$.
Since
\[1-\|P_B\|  = \inf \left\{\langle( I-P)f ,f \rangle_{\pi} /\|f\|_2:f \in \mathbb{R}^V, \eset \neq \mathrm{supp}(f) \subseteq B     \right\},\]
by \eqref{e:countable2} we clearly have  for all $u>0$ that\[\phi(u) \ge \inf \left\{\langle( I-P)f ,f \rangle_{\pi} /\|f\|_2: \eset \neq  \mathrm{supp}( f)\in \mathrm{FI}(V, u)     \right\} =\Lambda(u),\]  as $\phi(u)$ is obtained from $\Lambda(u)$ by restricting the collection of test functions considered in the middle term to indicators.     
By a well-known variant of Cheeger's inequality (cf.\ the proof of Lemma 2.4 in \cite{spectral})
\begin{equation}
\label{e:Chegerineqprofile}
\frac 12 \phi(u)^2 \le  \Lambda(u) \le \phi(u), \quad \text{for all} \quad u. 
\end{equation}
 The proof in \cite{spectral} is presented in the setup where $V$ is finite, but it is straightforward to extend it to the case that $V$ is countably infinite, using \eqref{e:countable2}.  Considering simple random walk on $\mathbb{Z}$ shows that it is possible that $\phi(u)^2 \asymp  \Lambda(u)$ for all $u$. 

Below we shall define quantities $\zeta_{\ell}(u)$ and $\chi_{\ell}(u)$. For a fixed $\ell \in \mathbb{N}$,  $\zeta_{\ell}(u)$  resemble $1/\Lambda(u)$. In fact, we will show that $\zeta_{\ell}(u) \le \frac{1}{c\Lambda \left( u/c\right)}$ for some  constant $c =c_{\ell}\in (0,1)$. Moreover, $\chi_{\ell}(u)$ will be obtained from $\zeta_{\ell}(u)$ by restricting to indicator test functions in a similar fashion as  $\phi(u)$ is obtained from  $\Lambda(u)$  by doing so. In particular, similarly to $\Lambda(u) \le \phi(u)$, we trivially have that  $\chi_{\ell}(u) \le \zeta_{\ell}(u)$ for all $\ell$ and all $u>0$. In light of Cheeger's inequality \eqref{e:Chegerineqprofile} and the aforementioned example of simple random walk on $\mathbb{Z}$, one might expect that an opposite inequality to   $\chi_{\ell}(u) \le \zeta_{\ell}(u)$ would involve  $\chi_{\ell}(u)^2$.   Interestingly, the following proposition asserts that for some absolute constants $\ell \in \mathbb{N}$ and $c\in (0,1)$, we have for all $u>0$ that \[ \frac{c}{\Lambda \left( u\right)} \le \chi_{\ell}(u) \le \zeta_{\ell}(u) \le \frac{1}{c\Lambda \left( u/c\right)} \le \frac{\chi_{\ell}(u/c)}{c^{2}}.  \] 
Surprisingly, $\chi_{\ell}(\cdot)$ does not appear with a square in the rightmost term.

We write   $\ell^p(\mathcal{P}(V) ,\pi)$ for the collection of all  distributions $\mu$ on $V$ satisfying that $\|\mu \|_{\pi,p}<\infty$, where as in the case that $V$ is finite, the $\ell^p$ norm $\| \sigma \|_{\pi,p}$ of a signed measure $\sigma$ on $V$ is defined as  $\| \sigma \|_{\pi,p}:=\| \frac{\sigma}{\pi} \|_{ p}$, where $\frac{\sigma}{\pi} \in \mathbb{R}^V$ is defined as $\frac{\sigma}{\pi}(v):=\frac{\sigma(v)}{\pi(v)}$ for all $v \in V$. 
For $t \ge 0$ let $P_t:=e^{t(P-I)}$. We write $\mu_t:=\mu P_t$ and $f_t:=P_tf$. For $\ell \in \N$ and $u>0$ let  
\begin{equation}
\label{e:0ell2decayviasp}
\begin{split}
\zeta_{\ell}(u):=\inf \{t: \|\mu_{t} \|_{2,\pi}^2 & \le 2^{-\ell} \|\mu    \|_{2,\pi}^2 \; \text{ for all } \; \mu \in \ell^2(\mathcal{P}(V) ,\pi) \, \; \text{such that } u\|\mu    \|_{2,\pi}^2 \ge 1  \}  \\
=\inf \{t: \|f_{t}\|_{2}^2 & \le 2^{-\ell}  \|f    \|_{2}^2 \; \text{ for all } \;  f\in \ell^2(V ,\pi)  \, \; \text{such that } u\|f    \|_{2}^2 \ge  \|f    \|_{1}^2   \}.
\end{split}
\end{equation}
The second equality follows from the fact that (by reversibility) $\|\mu P_{t}  \|_{2,\pi}^2  =\| P_{t} \frac{\mu}{\pi}  \|_{2}^2$  for all $\mu \in \mathcal{P}(V)$ and all $t \in \mathbb{R}_+$.  We also consider
for $\ell \in \N$ and $u>0$ \begin{equation}
\label{e:00ell2decayviasp}
\begin{split}
\chi_{\ell}(u): & =\inf \{t: \|\pi_B  P_{t}\|_{2,\pi}^2  \le 2^{-\ell}  \|\pi_B     \|_{2,\pi}^2 \; \text{ for all } \; B \in  \mathrm{FI}(V, u)   \}  \\ & =\inf \{t: \| P_{t}1_{B}\|_{2}^2  \le 2^{-\ell}  \|1_{B}\|_{2}^2 \; \; \text{ for all } \; B \in  \mathrm{FI}(V, u)   \}.
\end{split}
\end{equation}
We note that $\|\pi_B     \|_{2,\pi}^2=\frac{1}{\pi(B)}$ and that $\pi(B)\|1_{B}\|_{2}^2=\|1_{B}\|_{1}^2$ for all $B$ such that $\pi(B)<\infty$. Hence $\chi_{\ell}(u)$ is obtained by restricting the infimum in the definition of $\zeta_{\ell}(u)$ to distributions $\mu$ (respectively, to test functions $f$)  of the form $\pi_B$ (respectively, $1_B$)  for $B \in  \mathrm{FI}(V, u)$.
 
 \begin{proposition}
\label{p:spectral3} 
Consider the above setup and notation. Assume that $\pi(V)=\infty$. Then
\begin{equation}
\label{e:spb201}
\chi_{\ell}(u) \le \zeta_{\ell}(u) \le \frac{\ell \log 2}{ \Lambda \left(u 2^{\ell+2}\right)} \quad \text{for all }u>0 \text{ and } \ell \in \N.
\end{equation}
Moreover, there exist absolute constants $\ell \in \N$ and $c \in (0,1) $ such that for all $u>0$
\begin{equation}
\label{e:spb2}
\chi_{\ell}(u) \ge c/\Lambda \left( u\right) \ge c(\ell \log 2)^{-1} \zeta_{\ell}( 2^{-\ell-2}u) . 
  \end{equation}
\end{proposition}
\subsection{Organization of the paper}
In \S\ref{s:bd} we specialize Theorem \ref{thm:1} to the case of birth and death chains. In \S\ref{s:example} we present an illustrative example which was previously discussed. In \S\ref{s:background} we present some further background on quasi-stationary distributions and hitting times. We also prove Theorem \ref{thm:1} (given Theorem \ref{thm:2}) and \eqref{e:spectralprofiletwodefinitionsareequiv}. In \S\ref{s:thm1.2} we prove Theorem \ref{thm:2}. In \S\ref{s:thm3} we prove Theorem \ref{thm:3}. Finally, in \S\ref{s:extensions} we prove Theorem \ref{thm:LambdaPucountable} and Proposition \ref{p:spectral3}.  
\subsection{Acknowledgements}
The author would like to thank Lucas Teyssier for carefully reading an earlier draft of this work and providing  feedback which improved the presentation of the paper.  The author is supported by NSERC grants.

\section{A simple application for birth and death chains}
\label{s:bd}
Consider a birth and death chain on $[n]:=\{1,2,\ldots,n\}$. That is, $P(i,j)>0$ if $|i-j|=1$ and  $P(i,j)>0$ if $|i-j|>1$. Birth and death chains are always reversible. For $\delta \in (0,1)$ let \[x_{\delta}:=\max \{i \in [n] : \pi([i]) \le \delta \} \quad \text{and} \quad y_{\delta}:=\min \{i \in [n] : \pi(\{i,i+1,\ldots,n \}) \le \delta \} .\] We write $a \vee b := \max \{a,b\}$. Let \[t_*:=\max \{\mathbb{E}_{\pi_{[i]}}[T_{i+1}]: i \le x_{3/4} \} \vee \max \{\mathbb{E}_{\pi_{\{i,i+1,\ldots,n \}}}[T_{i-1}]: i \ge y_{3/4} \}  . \]

\begin{theorem}
\label{thm:bd}
In the above setup and notation there exists an absolute constant $C>0$ such that
\[\frac 14 t_{*}\le  \rel \le C t_{*}. \]
\end{theorem}

\begin{proof}
We first prove the first inequality. Note that for all $i \le x_{3/4}$ we have that $\pi([i]) \le 3/4$ and $\mathbb{E}_{\pi_{[i]}}[T_{i+1}]=\mathbb{E}_{\pi_{[i]}}[T_{[i]^{c}}]$. Similarly,   for all $i \ge  y_{3/4}$ we have that $\pi(\{i,i+1,\ldots,n \}) \le 3/4$ and that $\mathbb{E}_{\pi_{\{i,i+1,\ldots,n \}}}[T_{i-1}]=\mathbb{E}_{\pi_{\{i,i+1,\ldots,n \}}}[T_{[i-1]}]$. By \eqref{e:almost2} and \eqref{e:spectralprofiletwodefinitionsareequiv} \[t_* \le \frac{1}{\Lambda(3/4)} \le \frac{4}{\widehat{\Lambda}(3/4)} \le\frac{4}{\widehat{\Lambda}(1)}= 4 \rel. \]
We now prove that $  \rel \le C t_{*}$. By Theorem \ref{thm:1} it suffices to show that $t_{\mathrm{H}}^{\pi} \le t_*$. Generally, $ t_{\mathrm{H}}^\pi= \max_{B \,: \pi(B) \le 1/2, \, B \text{ is irreducible}}\mathbb{E}_{\pi_{B}}[T_{B^c}] $. In the current setup, every irreducible set $B$ is an interval. Then there exist some $1 \le i \le j \le n$ such that $B=[[i,j]]:=\{ \ell \in [n]  : i \le \ell \le j  \}$. Assume in addition that  $\pi(B) \le 1/2$. Since $\pi([[y_{3/4}-1,x_{3/4}+1]])>\frac 12$  we either have that $i \ge y_{3/4}$ or that $j \le x_{3/4}$. We argue that
\[\mathbb{E}_{\pi_{[[i,j]]}}[T_{[[i,j]]^c}]\le \mathbb{E}_{\pi_{[j]}}[T_{j+1}] \wedge  \mathbb{E}_{\pi_{\{i,i+1,\ldots,n \}}}[T_{i-1}] . \]
 Using this, we get that if $j \le x_{3/4}$ then
\[\mathbb{E}_{\pi_{[[i,j]]}}[T_{[[i,j]]^c}]\le \mathbb{E}_{\pi_{[j]}}[T_{j+1}] \le t_*, \]
whereas if  $i \ge y_{3/4}$ then
\[\mathbb{E}_{\pi_{[[i,j]]}}[T_{[[i,j]]^c}]\le \mathbb{E}_{\pi_{\{i,i+1,\ldots,n \}}}[T_{i-1}] \le t_*, \]
as desired. 

It remains to verify that $\mathbb{E}_{\pi_{[[i,j]]}}[T_{[[i,j]]^c}]\le \mathbb{E}_{\pi_{[j]}}[T_{j+1}] \wedge  \mathbb{E}_{\pi_{\{i,i+1,\ldots,n \}}}[T_{i-1}]$. We will only prove that $\mathbb{E}_{\pi_{[[i,j]]}}[T_{[[i,j]]^c}]\le \mathbb{E}_{\pi_{[j]}}[T_{j+1}] $, as the inequality  $\mathbb{E}_{\pi_{[[i,j]]}}[T_{[[i,j]]^c}]\le\mathbb{E}_{\pi_{\{i,i+1,\ldots,n \}}}[T_{i-1}] $ can be proved in an analogous manner. Clearly, for $\ell \in [[i,j]]$ we have that $\mathbb{E}_{\ell}[T_{[[i,j]]^c}] \le \mathbb{E}_{\ell}[T_{j+1}] $. Hence $\mathbb{E}_{\pi_{[[i,j]]}}[T_{[[i,j]]^c}]\le \mathbb{E}_{\pi_{[[i,j]]}}[T_{j+1}] $. For $\ell<i$, by the strong Markov property, together with the fact that if $X_0=\ell<i$ then $T_i<T_{j+1}$, we have that
\[\mathbb{E}_{\ell}[T_{j+1}] \le \mathbb{E}_{i}[T_{j+1}] \le \mathbb{E}_{\pi_{[[i,j]]}}[T_{j+1}].   \] 
It follows that $\mathbb{E}_{\pi_{[[i,j]]}}[T_{j+1}] \le \mathbb{E}_{\pi_{[j]}}[T_{j+1}]$, as claimed.
\end{proof}

\section{An illustrative example}
\label{s:example}
Consider a  $n$-vertex graph $G=(V,E)$ which has a vertex $o$ which has a loop and three other edges incident to it, while all other vertices are of degree 3. Denote the relaxation time of the simple random walk on $G$  by $\rel(G)$. We pick $G$ such that $\rel(G) \lesssim 1$.

 We now modify $G$ as follows. We delete the loop at $o$ and attach to  $o $ a path of length $r=\lceil n^{1/3} \rceil$. We denote the vertices belonging to the path by  $F$. We employ the convention that $o \notin F$. Denote the resulting graph by $\widehat G$. Consider simple random walk  $\mathbf{X}:=(X_k)_{k=0}^{\infty}$ on  $\widehat G$. We denote the corresponding probability distribution and expectation by $\mathbb{P}$ and $\mathbb{E}$. We denote the stationary distribution of $\mathbf{X}$ by $\pi$.

Let $\delta \in (0,1/2)$.  Let $A \subset V $ be a set satisfying $\delta n \le |A| \le (1-\delta) n$. In particular, $A \cap F = \eset$, i.e., $A$ does not contain a single vertex from the path attached at $o$.  Denote $B:=(V \cup F) \setminus A$ (i.e., $B$ is the complement of $A$ w.r.t.\ the graph $\widehat G$). Our goal is to show that there exist  constants $C(\delta)$ and $N_0(\delta)$ (depending only on $\delta$) and an absolute constant $C_1$ such that if $n \ge N_0(\delta$) then for all such $A$ we have that
\[\mathbb{E}_{\pi_{B}}[T_{A^c}] \le C(\delta) \quad \text{and} \quad \lambda(B)  \le \lambda(F) \le C_1n^{-2/3}  . \]
Moreover, the inequality $\mathbb{E}_{\pi_{B}}[T_{A^c}] \le C(\delta)$ does not require that $A \cap F = \eset$ (below we only analyze the case that $A \cap F = \eset$, but the case that $A \cap F \neq \eset$ poses no new obstacles, and the analysis below can be modified to treat it). 
The inequality $\lambda(B)  \le \lambda(F) $ follows from the fact that $F \subset B$ using the second line of \eqref{e:lambda(A)CF} (see also \eqref{e:monotonicity} below for a different argument). 

  We label the vertices of $F$  by the set $[r]:=\{1,2,\ldots,r\}$, such that $1$ is a neighbor of $o$ and such that $i$ and $i+1$ are adjacent for all $1 \le i < r$. Consider the distribution $\alpha^F$ on $F$ given by  $\alpha^{F}(i)=\frac{\sin(\frac{\pi i}{2r})}{\sum_{j=1}^{r}\sin(\frac{\pi j}{2r})}$ for all $i \in [r]$. It is easy to check that $\alpha^{F} P_F=\cos(\frac{\pi }{r}) \alpha^{F}$.  It follows from the Perron--Frobenius theorem together with the fact that   $\alpha^{F}(i)>0$ for all $i \in F$ and the fact that $P_F$ is irreducible,   that $\alpha^{F}$ is the quasi-stationary distribution of $F$. Hence $\lambda(F) = 1-\cos(\frac{\pi }{r})$.  Since  $1-\cos(\frac{\pi }{r}) = \Theta(r^{-2})$ we get that \[\lambda(B)  \le \lambda(F) \le C_1r^{-2} \le C_1n^{-2/3},   \]
 as desired.
We now show that $\mathbb{E}_{\pi_{B}}[T_{A^c}] \le C(\delta)$. 

We write $T_o$ for $T_{\{o\}}$. Let $D:=B\setminus F=V \setminus A$. Note that $D$ is the complement of $A$ in the  graph $G$. We have that
\begin{equation}
\label{e:3cases}
\mathbb{E}_{\pi_{B}}[T_{A}] \le \frac{\pi(F)}{\pi(B)}  \mathbb{E}_{\pi_{F}}[T_{A}]+  \mathbb{E}_{\pi_{D}}[T_{A} \mathbf{1}\{ T_A<T_o\}]+\mathbb{E}_{\pi_{D}}[T_{A} \mid T_A>T_o]\mathbb{P}_{\pi_{D}}[T_A>T_o].
\end{equation}
We now decompose the first and third terms on the right-hand side as follows,
\begin{equation}
\label{e:example0}
\mathbb{E}_{\pi_{F}}[T_{A}]=\mathbb{E}_{\pi_{F}}[T_{o}]+\mathbb{E}_{o}[T_{A}].
\end{equation}
\begin{equation}
\label{e:example00}
\mathbb{E}_{\pi_{D}}[T_{A} \mid T_A>T_o]=\mathbb{E}_{\pi_{D}}[T_{o} \mid T_A>T_o]+\mathbb{E}_{o}[T_{A}].
\end{equation}
By the above analysis, together with \eqref{e:almost2}, we have that 
\begin{equation}
\label{e:example000}
\mathbb{E}_{\pi_{F}}[T_{o}]=\mathbb{E}_{\pi_{F}}[T_{F^c}] \le \mathbb{E}_{\alpha^{F}}[T_{F^c}]=1/\lambda(F)  \asymp r^2.
 \end{equation}

Let $\mathbf{Y}:=(Y_k)_{k=0}^{\infty}$ be a simple random walk on $G$ (i.e., on the original graph, before attaching the path at $o$ and deleting the loop at $o$). We denote the corresponding probability distribution and expectation by $\mathbf{P}$ and $\mathbf{E}$. We denote the stationary distribution of $\mathbf{Y}$ by $\pi^G$. By construction we have  for all $v \in V$ that $v$ has the same degree in $G$ and in $\widehat G$ (this is because $o$ has one additional neighbor in $\widehat G$ but $o$ has a loop in $G$ and not in $\widehat G$). Hence $\pi_V=\pi^G$ and for all $W \subset V$ we have that $\pi_W=(\pi^G)_W$. 

 We will show that if $n \ge N_0(\delta)$ then
\begin{equation}
\label{e:example1}
\mathbb{E}_{o}[T_{A}] \lesssim_{\delta} r=\lceil n^{1/3} \rceil,
\end{equation}
\begin{equation}
\label{e:example2}
\mathbb{E}_{\pi_{D}}[T_{A} \mathbf{1}\{ T_A<T_o\}] = \mathbf{E}_{\pi_{D}}[T_{A} \mathbf{1}\{ T_A<T_o\}]    \le \mathbf{E}_{\pi_{D}}[T_{A} ]   \lesssim_{\delta} 1,
\end{equation}
\begin{equation}
\label{e:example3}
\mathbb{P}_{\pi_{D}}[T_A>T_o]=\mathbf{P}_{\pi_{D}}[T_A>T_o] \lesssim_{\delta}  \frac{\log n}{n},
\end{equation}
\begin{equation}
\label{e:example4}
\mathbb{E}_{\pi_{D}}[T_{o} \mid T_A>T_o] = \mathbf{E}_{\pi_{D}}[T_{o} \mid T_A>T_o] \lesssim_{\delta} \log n . 
\end{equation}
Combining \eqref{e:example0},   \eqref{e:example000} and \eqref{e:example1}  yields that if $n \ge N_0(\delta)$ then
\begin{equation*}
\mathbb{E}_{\pi_{F}}[T_{A}] \lesssim r^2 \asymp n^{2/3} \quad \text{and hence that} \quad \frac{\pi(F)}{\pi(B)}  \mathbb{E}_{\pi_{F}}[T_{A}]\lesssim_{\delta} 1.
\end{equation*}
Combining \eqref{e:example00}, \eqref{e:example1}, \eqref{e:example3} and  \eqref{e:example4} gives that
\begin{equation*}
\mathbb{E}_{\pi_{D}}[T_{A} \mid T_A>T_o]\mathbb{P}_{\pi_{D}}[T_A>T_o] =o(1).
\end{equation*}
Plugging in the last two displays and \eqref{e:example2} in \eqref{e:3cases} yields that  $\mathbb{E}_{\pi_{B}}[T_{A}]  \lesssim_{\delta} 1$ as desired.

To conclude the analysis of the example we now prove \eqref{e:example1}-\eqref{e:example4}. We first prove \eqref{e:example2}. The equality in \eqref{e:example2} follows from the observation that if $X_0 \in V$ (in particular, if $X_0 \sim \pi_D$) then $\{X_t:0 \le t \le T_o \} \subseteq V$ (i.e., $\{X_t:0 \le t \le T_o \} \cap F= \eset$). The first inequality in \eqref{e:example2} is trivial. The second inequality in \eqref{e:example2} follows from the fact that if $\alpha^{D}$ is the quasi-stationary distribution of the set $D$ for  $\mathbf{Y}$ (i.e.\ for simple random walk on $G$) then by  \eqref{e:almost2} and \eqref{e:spectralprofiletwodefinitionsareequiv} we have that
\[\mathbf{E}_{\pi_{D}}[T_{A} ] =\mathbf{E}_{(\pi^G)_{D}}[T_{A} ] \le \mathbf{E}_{\alpha^{D}}[T_{A}   ]  \le \rel(G)/\pi^{G}(A) \lesssim_{\delta}1.     \]
This concludes the proof of \eqref{e:example2}. We now prove \eqref{e:example3}.   The  equality in  \eqref{e:example3} follows in exactly the same manner as the equality in \eqref{e:example2}. We now show that $\mathbf{P}_{\pi_{D}}[T_A>T_o] \lesssim_{\delta}  \frac{\log n}{n}$.
We write $N(a,b):=\mathbf{E}_a[|\{k \in [0,T_A) \cap \mathbb{Z}:Y_k=b \}|] $ for the expected number of visits to $b$ made by $\mathbf{Y}$ before time $T_A$, when $Y_{0}=a$. Similarly, let $N_o:=\mathbf{E}_{\pi_D}[|\{k \in [0,T_A) \cap \mathbb{Z}:Y_k=o \}|]$. By reversibility and the fact that $(\pi^G)_D=\pi_D$ we get that $\pi_D(x)N(x,o)=\pi_D(o)N(o,x)$ for all $x\in D$. Hence
\[\mathbf{P}_{\pi_{D}}[T_A>T_o] \le N_o= \sum_{x \in D}\pi_D(x)N(x,o)= \sum_{x \in D}\pi_D(o)N(o,x) \lesssim_{\delta}  \frac{1}{n}\mathbf{E}_{o}[T_A] .   \]
Let $\mix(G,\eps)$ be the $\eps$ total variation mixing time of the rate-1 continuous time version of  $\mathbf{Y}$. It follows from  \cite{Aldoussome} (as well as from \cite{Oliveira} and\cite{PS}; see the discussion in page 2) that $\mathbf{E}_{o}[T_A] \lesssim_{\delta} \mix(G,1/4)$. It is standard (e.g.\ \cite[Theorem 20.6]{levin}) that $\mix(G,1/4) \lesssim \rel(G)  \log n$, and so  \[\mathbf{E}_{o}[T_A] \lesssim_{\delta} \mix(G,1/4) \lesssim \rel(G)  \log n \lesssim \log n.\] This concludes the proof of \eqref{e:example3}. We now prove \eqref{e:example4}. 

The  equality in  \eqref{e:example4} follows in exactly the same manner as the equality in \eqref{e:example2}. The inequality in  \eqref{e:example4} follows from the fact that $\mathbf{P}_{\pi_{D}}[T_A>T_o] \ge \pi_D(o) \ge \frac{1}{n}$, in conjunction with
\[\mathbf{P}_{\pi_{D}}[T_o =\ell,T_A  >\ell] \le \mathbf{P}_{\pi_{D}}[T_A > \ell] \le \mathbf{P}_{\alpha^{D}}[T_A > \ell] \le \exp \left(-\ell \pi^G(A)/\rel(G)  \right),     \]
where in the penultimate inequality we used \eqref{e:almost2} and in the last inequality we used \eqref{e:spectralprofiletwodefinitionsareequiv} together with the fact that $1-x \le e^{-x}$ for all $x \in \R$ (as well as that $\mathbf{P}_{\alpha^{D}}[T_A > \ell] =\left(1-\lambda_G(D) \right)^{\ell} \le \exp(- \ell \lambda_G(D) )$, where the subscript $G$ indicates that this is taken w.r.t.\ $\mathbf{Y}$). Indeed, since $ \pi^G(A)/\rel(G) \gtrsim \delta$, there exists an absolute constant $C'>0 $ such that
\[\sum_{\ell:\ell \ge C'\delta^{-1}\log n}\ell \exp \left(-\ell \pi^G(A)/\rel(G)  \right)\le 1/n^{2} \le \frac{\mathbf{P}_{\pi_{D}}[T_A>T_o]}{n} .  \]    

It remains to prove \eqref{e:example1}. We use a coupling argument. One way of sampling $(X_t)_{t=0}^{T_A}$ with $X_0=o$ is to first sample a copy of  $(Y_t)_{t=0}^{T_A}$ with $Y_0=o$, and then replace for all $i$ the $i$th time   $(Y_t)_{t=0}^{T_A}$ crosses the loop attached to $o$ in $G$  by an excursions from $o$ to itself that goes into the path $F$. The expected length of each such excursion is $r+1$. Since (as previously established) $\mathbf{E}_o[T_A] \lesssim_{\delta} \log n$, using Wald's identity, in order  to conclude the proof of \eqref{e:example1} it suffices to show that if $n \ge N_0(\delta)$ then uniformly over all $A$ such that $\delta n\le |A| \le (1-\delta)n$, we have that $\mathbf{P}_o[T_A<T_{\text{loop}}] \gtrsim 1$, where $T_{\text{loop}}=\inf \{k \in \mathbb{N}  :Y_{k}=o=Y_{k-1} \}$.

Let $T_{o}^+:=\inf \{k>0:Y_k=o \}$. Finally, we argue that  
\[\mathbf{P}_o[T_A<T_{\text{loop}}] \ge  \min_{v \in V}\mathbf{P}_o[T_v<T_{o}^+] \gtrsim 1/\rel(G) \gtrsim 1 .\]  The first and third inequalities in the last display are trivial. We now prove the middle one. Let $v \in V$. Let $K$ be the transition matrix of the simple random walk on $G$. Let $K_t:=e^{t(K-I)}$. By \cite[Corollary 2.8]{aldous} and \cite[Lemma 10.2]{levin},
\[\frac{1}{\min_{v \in V}\mathbf{P}_o[T_v<T_{o}^+] \pi^G(o)}=\mathbf{E}_o[T_v]+\mathbf{E}_v[T_o] \le 4 \max_{x \in V}\mathbf{E}_{\pi^G}[T_x] . \] 
By Lemma 2.11  in\footnote{See  \S2.2.3  in \cite{aldous} for an explanation of  the extension of the identity  $\pi^G(x)\mathbf{E}_{\pi^G}[T_x]=\sum_{\ell=0}^{\infty}\left(K^{ \ell}(x,x) -\pi^G(x \right)$  to its continuous-time counterpart $\pi^G(x)\mathbf{E}_{\pi^G}[T_x]=\int_0^{\infty}\left(K_t(x,x) -\pi^G(x)\right)\mathrm{d}t$). Alternatively, using the spectral decomposition it is easy to verify that the sum equals the integral.} \cite{aldous} 
\[ \max_{x \in V}\mathbf{E}_{\pi^G}[T_x] = \max_{x \in V}  \int_0^{\infty}\left(\frac{K_t(x,x) -\pi^G(x)}{\pi^G(x)}\right)\mathrm{d}t \le \frac{4n}{3}\int_0^{\infty}\exp \left(-\frac{t}{\rel(G)}\right)\mathrm{d}t  \lesssim n\rel(G),  \]
where the penultimate inequality follows from the fact that $\min_{x} \pi(x) \ge \frac{3}{4 n}$ together with the standard fact (often referred to as the Poincar\'e inequality) that $K_t(x,x) -\pi^G(x) \le (1-\pi^G(x) ) \exp \left(-\frac{t}{\rel(G)}\right) \le \exp \left(-\frac{t}{\rel(G)}\right)  $ for all $t \ge 0$ and all $x \in V$, which can easily be derived from the spectral decomposition (cf.\ \cite[Lemma 12.2]{levin}).  Combining the last two displays shows that $  \min_{v \in V}\mathbf{P}_o[T_v<T_{o}^+] \gtrsim 1/\rel(G)$, as desired. \qed

\section{Preliminaries}
\label{s:background}
Let $P$ be an irreducible reversible transition matrix on a finite state space $V$. Denote its stationary distribution by $\pi$. For a distribution $\mu$ and a function $f:V \to \mathbb{R} $ we write $\mathbb{E}_{\mu}[f]:=\sum_{u \in V}\mu(u)f(u)$ for its expectation w.r.t.\ $\mu$. Fix a set $\eset \neq B \subsetneq V$. Recall that we denote $\pi$ conditioned on $B$ by $\pi_B$. For $f,g:V \to \mathbb{R}$ we define the  inner product $\langle f,g \rangle_{\pi_{B}}:=\mathbb{E}_{\pi_{B}}[fg]$ and the corresponding $L^2$ norm  $\|f\|_{2,B}:=\sqrt{\langle f,f \rangle_{\pi_{B}}}$.  

 Recall that  $P_B(a,b)=P(a,b)\mathbf{1}\{a.b \in B\}$ and $C_0(B)=\{f \in \mathbb{R}^V:f(x)=0 \text{ for all }a \in B^c \}$. By  reversibility of $P$ we have that $P_B$ satisfies $\pi_B(a)P_B(a,b)=\pi_B(b)P_B(b,a)$ for all $a,b \in B$ and is thus self-adjoint w.r.t.\ the inner product $\langle \cdot ,\cdot \rangle_{\pi_{B}}$ on the space \[ C_0(B):=\{f \in \mathbb{R}^V:f(x)=0 \text{ for all }a \in B^c \}.\]
Hence $P_{B}$ has $k:=|B|$ real eigenvalues $ 1>\beta (B)=\gamma_1 \ge \gamma_2 \ge \cdots
\ge \gamma_{k} \ge - |\gamma_1|$ corresponding to eigenfunctions in $C_0(B)$, where the last inequality follows by the Perron--Frobenius theorem, and the first since $B$ is a proper subset of $V$. For the sake of ease of notation we write $\beta$ instead of $\beta(B)$ and  $\la:=1-\beta $ instead of $\lambda(B):=1-\beta (B)$. 

We now present some background on quasi-stationary distributions and their relation to hitting times. We begin the analysis by recalling (e.g.\ \cite[Ch.\ 4]{aldous}) that the quasi-stationary distribution $\alpha=\alpha^B$ corresponding to the set $B$ satisfies that the first hitting time of $B^c$ under $\mathbb{P}_{\alpha}$ has  a Geometric distribution of parameter $\lambda=1-\beta$.  

We now prove the above claim. It is an immediate consequence of the Perron--Frobenius theorem that there exists a distribution $\alpha=\alpha^B$ such that $\alpha P_B=\beta\alpha$ and so for this $\alpha$, \[\mathbb{P}_{\alpha}[T_{B^{c}}>m]=\alpha (P_B)^m(B)=\beta^{m},\] for all $m \in \mathbb{Z}_+$, where we used the fact that for all $x$,  \[\alpha (P_B)^m(x):=\sum_{b}\alpha (b)(P_B)^m(b,x)=\mathbb{P}_{\alpha}[X_{m}=x,T_{B^{c}}>t].\] It
follows that starting from initial distribution $\alpha$ the law of $T_{B^{c}}$ is Geometric  of parameter $\lambda$. Hence  $\mathbb{E}_{\alpha}[T_{B^{c}}]=1/\lambda$.

 The identity $\alpha (P_B)^m=\beta^{m} \alpha$ also implies that for all $x \in B$ and all $m\ge 0$
\[ \mathbb{P}_{\alpha}[X_{m}=x \mid T_{B^{c}}>m]=\beta^{-m} \mathbb{P}_{\alpha}[X_{m}=x , T_{B^{c}}>m]=\alpha(x). \] 
The last equation justifies the name ``quasi-stationary distribution".

We now derive the distribution of the hitting time of  $B^{c}$ starting from initial distribution $\pi_{B}$. When all of the eigenvalues of $P_B$ are non-negative, as is the case when $P(a,a) \ge 1/2$ for all $a$, this distribution is a mixture of Geometric distributions, whose parameters are the eigenvalues of $P_B$. However, we shall not assume that the eigenvalues of $P_B$ are all non-negative.   

Observe that for all $f \in C_0(B) $ and all $m \in \mathbb{N}$,  \[((P_B)^mf)(a):=\sum_b (P_B)^m(a,b)f(b)=\mathbb{E}_a[f(X_{m}) \mathbf{1}\{T_{B^{c}}>m \}].\]   
Let $f_1,\ldots,f_k$ be an orthonormal (w.r.t.\ the inner product $\langle \cdot , \cdot \rangle_{\pi_{B}}$)  basis of  $C_0(B)$, such that $P_B f_i=\gamma_i f_i$ for all $i$. By reversibility we can take  $f_1=\frac{\alpha/\pi_{B}}{\|\alpha/\pi_{B}\|_{2,B}}$.
Then $1_A=\sum_{i=1}^k c_{i}f_i $, where $c_i:=\mathbb{E}_{\pi_{B}}[f_i]$ and hence $(P_B)^m1_B=\sum_{i=1}^k c_{i}\gamma_i^m f_i $. Thus, by orthonormality,
\begin{equation}
\label{e:mixture}
\Pr_{\pi_{B}}[T_{B^{c}}>m]=\langle (P_B)^m1_B,1_B \rangle_{\pi_{B}}=\sum_{i=1}^{k} c_{i}^{2}\gamma_i^m = \frac{1}{\|\alpha/\pi_{B}\|_{2,B}^2}(1-\lambda)^m+\sum_{i=2}^k c_{i}^{2}\gamma_i^m .
\end{equation}
Taking $m=0$ we see that $\sum_{i=1}^kc_i^2=1$. In particular, since $\gamma_i \le \gamma_1$ for all $i$ we get that
\begin{equation}
\label{e:mixture2}
\frac{1}{\lambda \|\alpha/\pi_{B}\|_{2,B}^2} \le \mathbb{E}_{\pi_{B}}[T_{B^{c}}]\le \frac{1}{\lambda}.
\end{equation}
This implies the second inequality $\max_{D \,: \, D \subset B}\mathbb{E}_{\pi_{D}}[T_{D^c}] \le \frac{1}{\lambda(B)}$  from \eqref{e:2}, since 
\begin{equation}
\label{e:monotonicity}
\text{if } D \subset B \quad \text{then} \quad \lambda(D) \le \lambda(B),
\end{equation}
which follows from the second line of \eqref{e:lambda(A)CF}, or alternatively from the identity
\begin{equation}
\label{e:lambdaaslimit}
\beta(B)=\lim_{k\to \infty}\max_{x \in B} -\frac 1k \log \mathbb{P}_x[T_{B^c}>k] .
\end{equation}

\subsection{Proof of Theorem \ref{thm:1}, given Theorem \ref{thm:2}}
\label{s:thm1}
\begin{proof}
By \eqref{e:spectralprofiletwodefinitionsareequiv}  $(1-\delta)^{-1} \Lambda(\delta) \ge \widehat \Lambda(1)=\frac{1}{\rel} $.  Hence by \eqref{e:mixture2}  for all $\eset \neq A \subsetneq V$ we have that
\[\mathbb{E}_{\pi_{A^c}}[T_A] \le 1/\lambda(A^c) \le 1/\Lambda(\pi(A^c))  \le  \rel/\pi(A) . \]
This immediately gives the second inequality
$\max_{A \subsetneq V \, :\, \pi(A) \ge 1/2}\mathbb{E}_{\pi_{A^c}}[T_A] \le 2 \rel$ in \eqref{e:1}. 

We now prove the first inequality in \eqref{e:1}. Let $\beta=\beta_2=1-\frac{1}{\rel}$ be the second largest eigenvalue of $P$. Consider $f_2 \in \mathbb{R}^V $ such that $Pf_2=\beta f_2$ and $\|f\|_2=1$. Without loss of generality, $\pi(\{x:f_2(x)>0 \}) \le 1/2$, since otherwise we can consider $-f_2$ instead of $f_2$. Let $z$ be such that $f_2(z)=\max_{x}f_2(x)$. Let $A=\{x:f_2(x) \le 0 \}$. Then $Z_k=\beta^{-k \wedge T_{A}}f_2(X_{k\wedge T_{A}})$ is a martingale w.r.t.\ the natural filtration of the chain, where $a \wedge b:=\min(a,b)$. Since $f_2(x) \le 0$ for all $x \in A$ we have that $f_2(X_{T_{A}}) \le 0$, and so \[Z_k \le \beta^{-k }f_2(X_{k})\mathbf{1}\{ T_{A}>k\} \le \beta^{-k}f(z)\mathbf{1}\{ T_{A}>k\} .\]   Hence
\[f(z)=\mathbb{E}_z[Z_0] =\mathbb{E}_z[Z_k] \le   \beta^{-k  } \mathbb{P}_z[ T_{A}>k]f(z).  \]
Since $\mathbb{E}_{\pi}[f]=0$ and $\|f\|_2=1$ we must have that $f(z)$ is strictly positive. Rearranging the last display yields that $\liminf_{k\to \infty}-\frac 1k \log \mathbb{P}_z[T_{A^c}>k] \ge \beta=  1-\frac{1}{\rel}$. It follows from \eqref{e:lambdaaslimit} that $\lambda(A^c) \le \frac{1}{\rel} $. This concludes the proof, using the first inequality in \eqref{e:2}. \end{proof}
\subsection{The spectral profile - Proof of \eqref{e:spectralprofiletwodefinitionsareequiv}}
\label{s:1.9}
We now explain why \eqref{e:spectralprofiletwodefinitionsareequiv} holds. By the Courant--Fischer characterization of the eigenvalues of $P_B$, \begin{equation}
\label{e:lambda(A)CF}
\begin{split}
\lambda(B)&=\min\{\langle (I-P)f,f \rangle_{\pi}: \mathbb{E}_{\pi}[f^2] =1,f \in C_0(B)\}
\\ & = \min\{\langle (I-P)f,f \rangle_{\pi}: \mathbb{E}_{\pi}[f^2] =1,f \in C_0^{+}(B)\},
\end{split}
 \end{equation}
where the second equality is a consequence of the Perron--Frobenius theorem. Comparing the last display with \eqref{e:widehatlambdaBdef} we get that taking $f$ attaining the minimum in \eqref{e:lambda(A)CF} yields that  $\frac{\widehat \lambda(B)}{\lambda(B)} \ge \frac{\mathrm{Var}_{\pi}f}{\mathbb{E}_{\pi}[f^{2}]}$. A simple application of the Cauchy--Schwarz inequality yields that for all $f \in C_0(B)$ we have that
  $(\mathbb{E}_{\pi}[f])^{2}=(\mathbb{E}_{\pi}[f1_{B}])^{2} \le \pi(B)\mathbb{E}_{\pi}[f^{2}]$. Hence $\frac{\widehat \lambda(B)}{\lambda(B)} \ge \frac{\mathrm{Var}_{\pi}f}{\mathbb{E}_{\pi}[f^{2}]} \ge \pi(B^c)$, as desired.

\section{Proof of Theorem \ref{thm:2}}
\label{s:thm1.2}
Let $A$ be as above and $B=A^c$. For the remainder of this section we employ the following notation. Let $\alpha$ be the quasi-stationary distribution on $B$. Let 
 $f:=\alpha/\pi_{B}$. For $r \ge 0$ and for an interval $J \subset \mathbb{R}_+$ we write    \[B_r:=\{b\in B : f(b) \ge  r \}  \quad \text{and} \quad B_J:=\{b\in B : f(b) \in J \}.   \] 
We also write \[\pi_r:=\pi_{B_r} \quad  \text{and}   \quad \alpha_r \] for $\pi$ and $\alpha$ (respectively) conditioned on $B_r$ (i.e.\ $\pi_r:=\pi_{B_{r}}(x) $ and $\alpha_{r}(x):=\frac{\alpha(x) \mathbf{1}\{  x\in B_r \}}{\alpha(B_{r})}$). We also write for an interval $J \subset \mathbb{R}_+$, $\pi_J:=\pi_{B_J}$ and $\alpha_J$ for $\pi$ and $\alpha$ conditioned on $B_J$. 

We will prove the  first inequality from  \eqref{e:2} by considering a set of the form $ B_{L}$  for a careful choice of $L$.  The correct choice is subtle!

The constant $\frac{20}{17}$ in \eqref{e:lem3.1} is not optimal. For our purposes all that matters is that $L^{-1}\mathbb{E}_{\pi_{L}}[f]>  \frac{20}{17}$  is larger than one and that $\mathbb{E}_{\pi_{L}}[f^{2}] < 2(\mathbb{E}_{\pi_{L}}[f])^{2} $.

\begin{lemma}
\label{lem:good-superlevelset} For every  irreducible, reversible Markov chain on a finite state $V$ with stationary distribution $\pi$ and   every $\eset \neq B \subsetneq V$, there exists some $L>0  $ such that 
\begin{equation}
\label{e:lem3.1}
 \mathbb{E}_{\pi_{L}}[f]  = \frac{ \alpha(B_{L})}{\pi_{B}(B_{L})} >  \frac{20}{17 }L , \quad \text{and}
\end{equation}
\begin{align}
\label{e:lem3.12}
\mathbb{E}_{\pi_{L}}[f^{2}] = 2 L  \mathbb{E}_{\pi_{L}}[f] \le 4L^2  .
\end{align}
\end{lemma}
\noindent \emph{Proof.} Let $m_*:=\max_xf(x)$. Since for all $\ell \in (0,m_*] $ and $x \in B_{\ell} $ we have  that
\[\alpha(B_{\ell}) \alpha_{\ell}(x)f(x)=\alpha(x)^2/\pi_B(x)=\pi_B(x)f(x)^{2} =\pi_{B}(B_{\ell})\pi_{\ell}(x)f(x)^{2}\]
and that $\pi_{\ell}(x)f(x)=\frac{\alpha(x)}{\pi_{B}(B_{\ell})}$, we get that for all $\ell \in  (0,m_*] $,
\begin{equation}
\label{e:pialphamomenttranslation}
 \mathbb{E}_{\pi_{\ell}}[f]= \frac{\alpha(B_{\ell})}{\pi_{B}(B_{\ell})} \quad \text{and} \quad \mathbb{E}_{\pi_{\ell}}[f^{2}]=\frac{\alpha(B_{\ell})}{\pi_{B}(B_{\ell})}\mathbb{E}_{\alpha_{\ell}}[f]= \mathbb{E}_{\pi_{\ell}}[f]\mathbb{E}_{\alpha_{\ell}}[f].
\end{equation}
For $\ell \in (0,m_*]$ let
  \begin{equation}
\label{e:pialphamomenttranslationW2}
 U_{\ell}:=\frac{1}{\ell  }\mathbb{E}_{\alpha_{\ell}}[f  ]=\frac{\pi_{B}(B_{\ell})}{\ell \alpha(B_{\ell})  }\mathbb{E}_{\pi_{\ell}}[f^{2}]=\frac{1}{\ell\mathbb{E}_{\pi_{\ell}}[f]  }\mathbb{E}_{\pi_{\ell}}[f^{2}] ,
 \end{equation}
where the middle and third equalities are due to \eqref{e:pialphamomenttranslation}.

We label the set $f(B):=\{f(x):x \in B \}$ as $\{a_{i}:1 \le i \le k \}$ such that $a_i<a_{i+1}$ for all $i<k:=|f(B)|$.     
Observe that (recall that for an interval $J$, $\alpha_J$ means $\alpha$ conditioned on $B_J$) 
\begin{itemize}
\item[(i)]
The map $\ell \mapsto  U_{\ell}$ is continuous and strictly decreasing on $(a_{i-1},a_i )$ and is left-continuous at $a_i$ for all $i \in [  k]$, where for $m \in \N$ we write $[m]:=\{1,2,\ldots,m\}$. 
\item[(ii)] We  have  that $\lim_{\ell \to a_{i}^+} U_{\ell}=\frac{1}{a_{i}}\mathbb{E}_{\alpha_{(a_{i},m_*]}} f>\frac{1}{a_{i}}\mathbb{E}_{\alpha_{[a_{i},m_*]}} f =U_{a_{i}}$ for all $i\in [k-1] $.    
\end{itemize}
Let $r:=\frac {1}{3} \|f\|_{2,B}^2$.  Since $\|f\|_{1,B}=\mathbb{E}_{\pi_B}[|f|]=\alpha(B)=1$ and $f \ge 0$, by H\"older's inequality we get that \[m_*=\max_z |f(z)| \ge \|f\|_{2,B}^2/\|f\|_{1,B}=\|f\|_{2,B}^2 =\mathbb{E}_{\alpha}[f] .\]
 Since $f=\alpha/\pi_B$ and $B_r:=\{x:f(x)\ge r\}$ we have that
 \[3r=\|f\|_{2,B}^{2}=\mathbb{E}_{\pi_B}[f^2]=\mathbb{E}_{\alpha}[f] < \mathbb{E}_{\alpha}[f1_{B_{r}}]+r(1- \alpha(B_{r})) .   \]
Hence $\mathbb{E}_{\alpha}[f1_{B_{r}}] > (2+\alpha(B_{r}))r$, which clearly implies that $U_{\ell}>2$ for all $\ell \le r$. Since $U_{s}<2$ for all $s \in (m_*/2,m_*]$ it follows from (i) and (ii) above that        
 \[L:=  \max \{\ell \in [0,m_*]  :  U_{\ell}  \le  2\} \quad \text{is well-defined and that} \quad U_{L}=2.  \]    
By \eqref{e:pialphamomenttranslationW2} we get the equality in \eqref{e:lem3.12}. The inequality in \eqref{e:lem3.12} follows from the equality  in \eqref{e:lem3.12} by the Cauchy--Schwarz inequality.  We now argue that it must be the case that $\mathbb{E}_{\pi_{L}}[f] \ge \frac{20}{17}L $. Observe that $U_L=2$ and that the maximality of $L$ together with \eqref{e:pialphamomenttranslationW2} implies that

\[ \mathbb{E}_{\pi_{\ell}}[f^{2}] < 2 \ell\mathbb{E}_{\pi_{\ell}}[f]  \quad \text{ for all }\ell \in (L,m_*]. \]
 We get that
\begin{equation*}
\begin{split}
2L\mathbb{E}_{\pi_{L}}[f]& =\mathbb{E}_{\pi_{L}}[f^{2}]=\mathbb{E}_{\pi_{L}}[f^{2}1_{B_{[L,\frac{5L}{4})}}]+\mathbb{E}_{\pi_{L}}[f^{2}1_{B_{\frac{5L}{4}}}] \\ & \le \frac{5L}{4} \mathbb{E}_{\pi_{L}}[f1_{B_{[L,\frac{5L}{4})}}]+\pi_L(B_{\frac{5L}{4}}) \mathbb{E}_{\pi_{5L/4}}[f^{2}] \\
&  < \frac{5L}{4} \mathbb{E}_{\pi_{L}}[f1_{B_{[L,\frac{5L}{4})}}]+ \frac{5L}{2}\pi_L(B_{\frac{5L}{4}}) \mathbb{E}_{\pi_{5L/4}}[f]   \\
& =\frac{5L}{4} \mathbb{E}_{\pi_{L}}[f1_{B_{[L,\frac{5L}{4})}}]+ \frac{5L}{2} \mathbb{E}_{\pi_{L}}[f1_{B_{\frac{5L}{4}}}].
\end{split}
\end{equation*}
It follows that $ 3 \mathbb{E}_{\pi_{L}}[f1_{B_{[L,\frac{5L}{4})}}] < 2 \mathbb{E}_{\pi_{L}}[f1_{B_{\frac{5L}{4}}}]$ and thus that
\[  \frac{3}{5} \mathbb{E}_{\pi_{L}}[f] <\mathbb{E}_{\pi_{L}}[f1_{B_{\frac{5L}{4}}}].  \]
 Hence
\[\mathbb{E}_{\pi_{L}}[f-L] \ge  \mathbb{E}_{\pi_{L}}[(f-L)1_{B_{\frac{5L}{4}}}] \ge \frac{1}{4}\mathbb{E}_{\pi_{L}}[f1_{B_{\frac{5L}{4}}}] >\frac{3}{20}\mathbb{E}_{\pi_{L}}[f].   \]
It follows that $\mathbb{E}_{\pi_{L}}[f] > \frac{20}{17}L$, as desired.  
\qed 

\medskip

\noindent \emph{Proof of Theorem \ref{thm:2}.}
For the remainder of the proof we abbreviate and write $\beta=1-\la $   instead of $\beta(B)=1-\lambda(B)$. Let $L$ be as above.
Recall that for $a,b \in \mathbb{R}$ we write $a \wedge b:=\min \{a,b\}$.  Denote 
\begin{equation}
\label{e:p=psubt}
\begin{split}
  T_L=T_{B_{L}^c}:=\inf \{k \ge 0  :X_k \in B_{L}^c \}. 
\end{split}
\end{equation} 
Consider the martingales \[N_k:=\beta^{-k }f(X_{k }) \quad \text{and} \quad W_k:=N_{k \wedge T_{L}   }.  \]
Let $c \in (0,1)$ be some absolute constant, to be determined later. Let
\[t:=\lceil c/(4\lambda)  \rceil \quad \text{and} \quad p:= \mathbb{P}_{\pi_{L}}[T_{L } \le t ].  \]
The  first inequality in  \eqref{e:2} is immediate when $\lambda \ge \frac{1}{4} c$, since $\mathbb{E}_{\pi_{D^c}}[T_D] \ge 1$, and then we may take $c_1=\frac{1}{4} c $. Hence we assume that $\lambda \le \frac{1}{4} c$. Thus  $\beta= 1-\lambda \ge e^{-2\lambda} $ and 
\begin{equation}
\label{e:lambdatsmall}
 \beta^{-t}\le e^{2\lambda t} \le e^c. 
\end{equation} 
  We shall show that  
\begin{equation}
\label{e:EWt}
 \mathbb{E}_{\pi_{L}}[W_{t}]\le Le^{c}( p+2\sqrt{1-p} ). 
\end{equation}
Before proving this, let us explain how this  implies the assertion of Theorem \ref{thm:2}. By \eqref{e:lem3.1} and the fact that $(W_k)_{k \in \mathbb{Z}_+}$ is a martingale, we get that
\begin{equation}
\label{e:3.22}
\mathbb{E}_{\pi_{L}}[W_{t}]=\mathbb{E}_{\pi_{L}}[W_{0}]=\mathbb{E}_{\pi_{L}}[f] \ge  \frac{20}{17 }L. 
\end{equation}
Combining this with \eqref{e:EWt}, and picking $c$ such that $e^{c} \le \frac{20}{19}$,   yields that
\[  \frac{19}{17 }\le p+2\sqrt{1-p} .\] 
 Hence
\[(1-\sqrt{1-p} )^{2}=(1-  p) -2\sqrt{1-p} +1\le \frac{15}{17} .\] Thus 
\[1-p\ge \left(1-\sqrt{15/17}\right)^2 . \]
This implies that $\mathbb{E}_{\pi_{L}}[T_{L } ] \ge \left(1-\sqrt{15/17}\right)^2 t$, as desired.

We now prove \eqref{e:EWt}. Observe that  
$f(X_{t \wedge  T_{L} } ) \mathbf{1}\{ T_{L} \le t\} < L  \mathbf{1}\{ T_{L} \le t\}. $
 By \eqref{e:lambdatsmall} this implies  that
  \begin{equation}
\label{e:clearly2}
  \mathbb{E}_{\pi_{L}}[W_{t  } \mid T_{L } \le t ] \le e^{c}L \quad \text{and so} \quad \mathbb{E}_{\pi_{L}}[W_{t  }   \mathbf{1}\{ T_{L} \le t\} ] \le e^{c}Lp.
  \end{equation}    
 By the Cauchy--Schwarz inequality (writing $W_{t} \mathbf{1}\{T_{L } > t\}=W_{t} \mathbf{1}\{T_{L } > t\}^2$)
   \begin{equation}
\label{e:clearly222}
\begin{split}
\left(  \mathbb{E}_{\pi_{L}}[W_{t} \mathbf{1}\{T_{L } > t\}   ] \right)^2 & \le   \mathbb{E}_{\pi_{L}}[W_{t}^2 \mathbf{1}\{T_{L } > t\}   ] \mathbb{P}_{\pi_{L}}[T_{L } >t] \\ & \le(1-p) e^{2c}\mathbb{E}_{\pi_{L}}[f(X_t)^2 \mathbf{1}\{T_{L} > t\}   ].
\end{split}
  \end{equation}
Since  $\mathbb{P}_{\pi_L}[X_t=x,T_{L } > t] \le \mathbb{P}_{\pi_L}[X_t=x] \le \pi_L(x)$ for all $x \in B_L$, by  \eqref{e:lem3.12} we also have that  \[\mathbb{E}_{\pi_{L}}[f(X_t)^2 \mathbf{1}\{T_{L} > t\}   ]\le \mathbb{E}_{\pi_{L}}[f^2] \le  4L^{2}.  \]
Together with \eqref{e:clearly2} and \eqref{e:clearly222} this implies \eqref{e:EWt}. \qed

\section{Proof of Theorem \ref{thm:3}}
We begin this section with a simple lemma, which implies \eqref{e:introdefotrelgeom'}.  
\begin{lemma}
\label{lem:Geomsubmultiplicativity}
Let $m,k \ge 1$.  Denote $Q=P^{G(m)} $.  Then 
\begin{equation}
\label{e:L2normofQQ*4}
P^{G(mk)} =\frac{1}{k}Q+\frac{k-1}{k}P^{G(mk)}Q=\sum_{j=1}^{\infty}\frac{1}{k}\left( \frac{k-1}{k}\right)^{j-1}Q^j.
\end{equation}
Moreover,
 \begin{equation}
\label{e:L2normofQQ*5} 
\|P^{G(mk)} - \Pi \| \le \frac{\|P^{G(m)} - \Pi\|}{k-(k-1)\|P^{G(m)} - \Pi\| }.
\end{equation}
\end{lemma}
\begin{proof}
It is easy to check that if  $Z,\eta_m$ and $\eta_{mk}$ are independent,  $1+\eta_{\ell} \sim \text{Geom}(1/\ell)$ for $\ell \in \{m,mk \}$   and $Z \sim \text{Bern}(1-\frac{1}{k})$  then $1+\eta_m+\eta_{mk}Z\sim \text{Geom}(\frac{1}{mk})$. Hence
\[P^{G(mk)}=\sum_{j \in \Z_+}P^j\mathbb{P}[\eta_m+Z\eta_{mk}=j]=\frac{1}{k}\sum_{j \in  \Z_+}P^j\mathbb{P}[\eta_m=j]+\frac{k-1}{k}\sum_{j \in  \Z_+}P^j\mathbb{P}[\eta_m+\eta_{mk}=j]. \]
Since $\eta_m$ and $\eta_{mk}$ are independent, we have that $\sum_{j \in   \Z_+}P^j\mathbb{P}[\eta_m+\eta_{mk}=j]=P^{G(mk)}P^{G(m)}$. Since we also have that $\sum_{j \in   \Z_+}P^j\mathbb{P}[\eta_m=j]=P^{G(m)}$, this gives the first equality in \eqref{e:L2normofQQ*4}. The second equality in \eqref{e:L2normofQQ*4} follows from applying the first one iteratively. By \eqref{e:L2normofQQ*4} \[P^{G(mk)} - \Pi =\frac{1}{k}(Q- \Pi )+\frac{k-1}{k}(P^{G(mk)}- \Pi )(Q- \Pi ).\]  By the triangle inequality and submultiplicativity of the operator norm, it follows  that
\[\|P^{G(mk)} - \Pi\| \le \frac{1}{k}\|P^{G(m)} - \Pi\|+\frac{k-1}{k}\|P^{G(m)} - \Pi\|\|P^{G(mk)} - \Pi\|.    \]
Rearranging yields  \eqref{e:L2normofQQ*5}.
\end{proof}
Let $Q$ be an irreducible transition matrix whose stationary distribution is $\pi$. For $\eset \neq  B \subsetneq V$ we have that $C_0(B)=\{f \in \mathbb{R}^V:\mathrm{supp}(f) \subset B \} $ is  invariant under $Q_B$ (when we interpret $Q_B$ as being labeled by $V$ and defined as $Q_B(x,y)=Q(x,y)\mathbf{1}\{x,y \in B\}$, as opposed to being labeled by $B$). We  consider the operator norm of $Q_B$, given by 
\begin{equation}
\label{e:introdefofoperatornorm3}
\|Q_{B}\|:=\max \{\|  Q_{B}f\|_{2}  : f \in \mathbb{R}^V, \|f \|_{2}=1  \}.
\end{equation}  
It is easy to see that the maximum in \eqref{e:introdefofoperatornorm3}   is attained by some non-negative $f \in C_0(B)$. However, when $Q$ is non-reversible, it need not be the case that this $f$ is an eigenfunction of $Q_B$. Recall that $\beta_i(K)$ denotes the $i$th largest eigenvalue of $K$.
It is standard that
\begin{equation}
\label{e:introdefofoperatornorm4}
\begin{split}
\|Q_{B} \|^2=\|Q_B^*Q_B\|= \beta _{1}(Q_B^*Q_B)=\beta_{1}(Q_BQ_B^*)=\|Q_BQ_B^*\|=\|Q_B^* \|^2 \le \|(Q^*Q)_{B} \|.
\end{split}
\end{equation}
By \eqref{e:spectralprofiletwodefinitionsareequiv} (together with $ 1-\|(Q^*Q)_{B} \|=\lambda_{Q^*Q}(B)$) we have that
\begin{equation}
\label{e:introdefofoperatornorm4'}
\begin{split}
 1-\|(Q^*Q)_{B} \| \ge \Lambda_{Q^*Q}(\pi(B))  \ge \pi(B^c) \widehat\Lambda_{Q^*Q}(\pi(B) )\ge \pi(B^c)(1- \|Q^*Q - \Pi \|).
\end{split}
\end{equation}

\subsection{Proof of Theorem \ref{thm:3}}
\label{s:thm3}
\begin{proof}
We first prove the second inequality in \eqref{e:3T3}.    Let $s:=\rel^{\text{Geom}}(1/e)$. We   write $Q:=P^{\mathrm{G}( s)}$. By the definition of $\rel^{\text{Geom}}(1/e)$  we have that $\beta_2(Q^{*}Q) = 1/e^2$.  Let $A\subsetneq V$ be such that $\pi(A) \ge 1/2 $. Denote $B:=A^c$. Let $b:=\frac{1+e^{-2}}{2}$. By \eqref{e:introdefofoperatornorm4} and \eqref{e:introdefofoperatornorm4'}
\[ 1-\|Q_{B}\|^{2} \ge 1-\|(Q^*Q)_{B} \|  \ge \pi(A)(1- e^{-2}) \ge \frac{1}{2}(1-e^{-2}), \quad \text{and so} \quad \|Q_{B}\| \le \sqrt{b}. \]
Consider $(\eta_s(j):j \in \N )$ which are i.i.d.\ such that $1+\eta_s(1) \sim$ Geom($1/s$). Let \[\rho :=\inf\{k:X_{\tau_k} \in A \}, \quad \text{where} \quad \tau_k:=\sum_{j=1}^k\eta_s(j). \] Since $\|(Q_{B})^m\| \le\|Q_{B}\|^{m} $ and $\langle (Q_{B})^{m}1_B,1_B \rangle_{\pi} \le\|(Q_{B})^m\|\langle1_B,1_B \rangle_{\pi} \le \pi(B)\|Q_{B}\|^{m}$, we have \[\pi(B)\mathbb{E}_{\pi_{B}}[\rho>m] =\langle (Q_{B})^{m}1_B,1_B \rangle_{\pi} \le \pi(B) \|Q_{B}\|^{m}.  \] Hence $\mathbb{E}_{\pi_{B}}[\rho] \le \frac{1}{1-\|Q_{B}\|} \le\frac{1}{1-\sqrt{b} }$. By Wald's identity $\mathbb{E}_{\pi_{B}}\left[\tau_{\rho}\right] = \mathbb{E}_{\pi_{B}}[\rho](s-1)$. We get that
\[\mathbb{E}_{\pi_{B}}[T_{A}] \le \mathbb{E}_{\pi_{B}}\left[\tau_{\rho}\right] \le \frac{s-1}{1-\sqrt{b} } , \]
as desired.
We now prove the first inequality in \eqref{e:3T3}.

Let $c:=1/\sqrt{20}$. Let $t:=\rel^{\text{Geom}}(1-c)$. Let $K:=P^{G(t)}$. Let $\eta_t$ be independent of $(X_k)_{k \in \mathbb{Z}_+}$ and satisfy $1+\eta_t \sim \text{Geom}(1/t)$. Let $C \ge 1$ be some absolute constant, to be determined soon. We may assume that $t \ge C$, since otherwise \eqref{e:introdefotrelgeom'} implies that  $\rel^{\text{Geom}}(1/e)$ is bounded above by some absolute constant, in which case we may use the trivial estimate $\mathbb{E}_{\pi_D}[T_{D^c}] \ge 1$. It suffices to show that there exists an absolute constant $\delta \in (0,1)$ such that  for some  $D \subsetneq V$ satisfying that $\pi(D) \le 1/2$ we have that 
\begin{equation}
\label{e:NTS}
\mathbb{P}_{\pi_{D}}[T_{D^{c}}>\eta_t  ]  \ge \delta.
\end{equation}
Indeed, provided that $C$ is chosen to be sufficiently large, there is an absolute constant $c' \in (0,1/2)$ such that $\mathbb{P}[\eta_t < c ' \delta t  ] \ge \delta / 2$, which in conjunction with \eqref{e:NTS} implies that $\mathbb{P}_{\pi_{D}}[T_{D^{c}}>c ' \delta t  ]  \ge \delta/2$. Hence $\mathbb{E}_{\pi_{D}}[T_{D^{c}}  ] \ge \frac 12 c ' \delta^{2} t$, which by  \eqref{e:introdefotrelgeom'} yields that $\mathbb{E}_{\pi_{D}}[T_{D^{c}}  ] \gtrsim \rel^{\text{Geom}}(1/e) $, as desired.

We now prove \eqref{e:NTS} (for some $D$ such that  $\pi(D) \le 1/2$  to be defined below). By \eqref{e:introdefofoperatornorm2} and  \eqref{e:introdefotrelgeom} we have that $\beta_2(KK^*)=1-c^2$. Let $f \in \mathbb{R}^V$ be an eigenfunction of $KK^*$ corresponding to $\beta_2(KK^*)$, i.e.,  $KK^* f=(1-c^2)f$.  Let $B:=\{x \in V:f(x)>0 \}$. Without loss of generality we may assume that $\pi(B) \le 1/2$, since otherwise we may consider $-f$ instead of $f$.
We normalize $f$ such that $f = \frac{\mu}{\pi_B}$, where $\mu 1_B$ is a distribution on $B$. This can always be done since if $Q$ is reversible w.r.t.\ $\pi$, $a \in [-1,1]$ and $\sigma$ is a signed measure on $V$ then $\sigma Q=a\sigma $ if and only if $Q \frac{\sigma}{\pi}=a \frac{\sigma}{\pi}$ (where, as usual, $\frac{\sigma}{\pi} \in \mathbb{R}^V$ is given by $\frac{\sigma}{\pi} (x):=\frac{\sigma(x)}{\pi(x)}$). In our case $Q:=KK^*$ and the assertion of the last sentence  means that one can write the right eigenfunction $f$ as $f=\frac{\nu}{\pi}$ where $\nu KK^*=(1-c^2)\nu$. Then $\nu$ is positive on $B$ (since $f(x)$ and $\nu(x)$ have the same sign for all $x \in V $) and if we normalize $\nu$ (and hence $f$) so that $\sum_{b \in B}\nu(b)=\pi(B)$ and set $\mu:=\nu/\pi(B)$, then we indeed have that $f=\frac{\mu}{\pi_B}$ and that  $\mu 1_B$ is a distribution on $B$. 

 We employ a similar notation as in \S\ref{s:thm1.2}. For the remainder of this section, we write   \[B_r:=\{b\in B : f(b) \ge  r \}.  \] 
We also write $\pi_r:=\pi_{B_r}$ for $\pi$ conditioned on $B_r$. 
Analogously to Lemma \ref{lem:good-superlevelset}, we have the following lemma. Since the proof is completely analogous, we omit it.

\begin{lemma}
\label{lem:good-superlevelset'} There exists some $L>0  $ such that 
\begin{equation}
\label{e:lem3.1'}
 \mathbb{E}_{\pi_{L}}[f]  = \frac{ \mu(B_{L})}{\pi_{B}(B_{L})} \ge  \frac{20}{17 }L  \quad \text{and}
\end{equation}
\begin{align}
\label{e:lem3.12'}
\mathbb{E}_{\pi_{L}}[f^{2}] = 2 L  \mathbb{E}_{\pi_{L}}[f] \le 4L^2  .
\end{align}
\end{lemma}
Fix $L$ as in Lemma \ref{lem:good-superlevelset'}. Recall that  $t=\rel^{\text{Geom}}(1-c)$ and  $c=1/\sqrt{20}$.  We now prove that \eqref{e:NTS} holds for $D=B_L$ with $\delta= \left(1-\sqrt{15/17}\right)^2 $.
Let $1+\eta_t$ and $1+\eta_t^*$ be i.i.d.\ Geom$(1/t)$ and independent of  $(X_k)_{k \in \mathbb{Z}_+}$. Let $(Y_k)_{k \in \mathbb{Z}_+}$ be defined as follows: Until time $\eta_t$ it evolves like a Markov chain with transition matrix $P$. After time $\eta_t$ it evolves like a Markov chain with transition matrix $P^{*}$ (with initial state $Y_{\eta_t}$), which is independent of $\eta_t^*$. That is,   $(Y_k)_{k=0}^{\eta_t}=(X_k)_{k=0}^{\eta_t}$ and for all $v \in V$ and $m \in \mathbb{Z_+}$, given that $(X_{\eta_t},\eta_t)=(v,m)$, writing $Z_k:=Y_{k+m}$ we have that $(Z_k)_{k \in \mathbb{Z}_+}$ is independent of $\eta_t^*$ and evolves like the Markov chain with transition matrix $P^*$ with initial state $v$. This construction gives that
\begin{equation}
\label{e:YetaKKstar}
\mathbb{E}_{x}[f(Y_{\eta_t+\eta_t^*})]=KK^*f(x)=(1-c^{2})f(x) \quad \text{for all }x \in V.
\end{equation}
Denote $T_L:=T_{B_{L}^c}=\inf \{k:Y_k \in B_L^c \}$. Denote
\[p:=\mathbb{P}_{\pi_L}[T_{L} \le \eta_t]+\mathbb{P}_{\pi_L}[T_{L} > \eta_t,Y_{\eta_t+\eta_t^*} \notin B_L ].\]
While we are only interested in proving that $\mathbb{P}_{\pi_L}[T_{L} \le \eta_t]$ is bounded away from 1, it will be convenient to instead argue that $1-p$ is bounded away from 0.

By the memoryless property of the Geometric distribution, for all $x \in B_L^c$ we have that 
\[\mathbb{E}_{\pi_L}[f(Y_{\eta_t+\eta_t^*}) \mid T_{L} \le \eta_t  ,X_{T_{L}}=x]=\mathbb{E}_{x}[f(Y_{\eta_t+\eta_t^*})]=(1-c^{2})f(x) < (1-c^2)L < L  . \]
Since $f(Y_{\eta_t+\eta_t^*})\mathbf{1}\{T_{L} > \eta_t  ,Y_{\eta_t+\eta_t^*} \notin B_L\} < L \mathbf{1}\{T_{L} > \eta_t  ,Y_{\eta_t+\eta_t^*} \notin B_L\}$, it follows that
\begin{equation}
\label{e:exitearly}
\mathbb{E}_{\pi_L}[f(Y_{\eta_t+\eta_t^*})\mathbf{1}\{T_{L} \le \eta_t  \}]+\mathbb{E}_{\pi_L}[f(Y_{\eta_t+\eta_t^*})\mathbf{1}\{T_{L} > \eta_t  ,Y_{\eta_t+\eta_t^*} \notin B_L\}]<pL. \end{equation}
Denote $\xi:=\mathbf{1}\{T_{L} > \eta_t ,Y_{\eta_t+\eta_t^*} \in B_L\}$. Since  $f(Y_{\eta_t+\eta_t^*})\xi=f(Y_{\eta_t+\eta_t^*})\xi^2$ and $\mathbb{E}_{\pi_{L}}[\xi]=1-p$,  the Cauchy--Schwarz inequality and \eqref{e:lem3.12'}
yield that\begin{equation*}
\begin{split}
\left( \mathbb{E}_{\pi_{L}}[f(Y_{\eta_t+\eta_t^*})\mathbf{1}\{T_{L} > \eta_t ,Y_{\eta_t+\eta_t^*} \in B_L\}   ]\right)^2 & \le \mathbb{E}_{\pi_{L}}[f(Y_{\eta_t+\eta_t^*})^{2}\mathbf{1}\{ Y_{\eta_t+\eta_t^*} \in B_L\}](1-p)
\\ & \le (1-p) \mathbb{E}_{\pi_L}[f^2] \le 4L^2(1-p),
\end{split}
\end{equation*}
where in the second inequality we used the fact that
\[\mathbb{P}_{\pi_{L}}[ Y_{\eta_t+\eta_t^*} =x] =\frac{\mathbb{P}_{\pi}[ Y_{\eta_t+\eta_t^*} =x,Y_{0} \in  B_{L}]}{\pi(B_L)} \le \frac{\mathbb{P}_{\pi}[ Y_{\eta_t+\eta_t^*} =x]}{\pi(B_L)}= \pi_L(x) \quad \text{for all }x \in B_L. \] 
Combining \eqref{e:lem3.1'}, \eqref{e:YetaKKstar}, \eqref{e:exitearly} and the display following \eqref{e:exitearly}  yields that 
   \begin{equation}
\label{e:clearly2222}
\frac{20}{17 }L  \le \mathbb{E}_{\pi_{L}}[f]=(1-c^{2})^{-1}\mathbb{E}_{\pi_{L}}[f(Y_{\eta_t+\eta_t^*})]<(1-c^{2})^{-1}[pL+2L \sqrt{1-p} ].
  \end{equation}
By our choice of $c$ we have that $1-c^{2}=\frac{19}{20}$. As in the proof of Theorem \ref{thm:2} we get that  
\[1-p\ge \left(1-\sqrt{15/17}\right)^2 , \]
as desired.
 \end{proof}

\section{Extensions of Theorems \ref{thm:1} and \ref{thm:2}}
\label{s:extensions}
\subsection{A discussion about the validity of our results in the continuous time setup}
\label{s:ctstime}
We now explain that all of the results from this paper are valid also for continuous-time Markov chains on a finite state space, and that this follows from the discrete time result. Indeed, the time $t$ transition probabilities of   a continuous-time Markov chain on a finite state space are given by the matrix $e^{tG}$, where $G$ is the Markov infinitesimal  generator of the chain. When the state space is finite, it is always the case that $G$ can be written as $r(P-I)$ for some $r \in (0,\infty)$ and some transition matrix $P$ (crucially, we allow the diagonal entries of $P$ to be positive). Moreover, $G$ is reversible if and only if $P$ is reversible. Since all considered quantities differ by exactly a multiplicative factor $r$ between $G$ and $P$, we get that the continuous time analogs of our results follow from the discrete time setup.
\subsection{Proof of Theorem \ref{thm:LambdaPucountable}}
\begin{proof}
We first note that  \eqref{e:countable2} follows from \eqref{e:countable1} in a straightforward manner. We now prove \eqref{e:countable1}. Let $B \subset V$ be such that $0<\pi(B)<\infty$. Let $B_1 \subseteq B_2 \subseteq \cdots$ be an exhaustion of $B$ (i.e., this is a nested sequence of subsets of $B$ satisfying that  $B=\cup_{n \in \N}B_n$) by finite sets. Since for all  $f \in  \ell^2(V,\pi)$ we have that $\|P_B f\|_2=\|P_B (1_{B}f)\|_2 \le\|P_B (1_{B}|f|)\|_2$ and $\|f\|_2=\||f|\|_2 \ge \|1_{B}|f|\|_2  $  we get that
\[\|P_B\|=\sup \{\|P_B f\|_2 /\|f\|_2: f \in  \ell^2(V,\pi), \eset \neq \mathrm{supp}(f) \subseteq B, f \ge 0 \}.\] 
Hence $\|P_B\| \ge \|P_D\|$ for all $D \subseteq B$.
  Since for all finite $D \subseteq B$ we have that $\|P_D\|=\beta(D)=\max \{\beta(D'):D' \subset D,\, D' \text{ is irreducible} \}=\max \{\|P_{D'}\|:D' \subset D,\, D' \text{ is irreducible} \}$, it suffices to show that  $\lim_{n \to \infty}\|P_{B_n}\|_2 \ge\|P_B\|$ (the limit exists since $\|P_{B_n}\|_2$ is non-decreasing).

Let  $f \in  \ell^2(V,\pi)$ be such that $\eset \neq \mathrm{supp}(f) \subseteq B$ and $f \ge 0$. Let $f_n:=1_{B_n}f $. Then $0 \le P_{B_n} f_{n}(x) \nearrow  P_{B}f(x)$ for all $x$. Hence  by the monotone convergence theorem $\lim_{n \to \infty}\|P_{B_n} f_{n}\|_2=\|P_B f\|_2$ and $\lim_{n \to \infty}\|f_{n}\|_2=\|f\|_2$. It follows that $\frac{\|P_B f\|_2 }{\|f\|_2}=\lim_{n \to \infty}\frac{\|P_{B_n} f_{n}\|_2}{\|f_{n}\|_2} \le \lim_{n \to \infty}\|P_{B_n}\|_2  $. Since $f$ is arbitrary, we get that $\|P_B \| \le\lim_{n \to \infty}\|P_{B_n}\|_2  $, as desired. This concludes the proof of \eqref{e:countable1}.  

We now prove \eqref{e:countable3}. In particular, we now also assume that $B$ is irreducible. Hence, we may pick an exhaustion  $B_1 \subseteq B_2 \subseteq \cdots$ of $B$ by finite irreducible sets.  Consider $g(b):=\mathbb{E}_{b}T_{B^c} $ and  $g_n(b):=\frac{\pi_{B_{n}}(b)}{\pi_B(b)} \mathbb{E}_{b}T_{B_{n}^c}$ (for $b \notin B_n$ we set $g_n(b):=0$). Then \[\mathbb{E}_{\pi_{B}}T_{B^c} =\sum_{b \in B}\pi_B(b)g(b) \quad \text{and} \quad \mathbb{E}_{\pi_{B_{n}}}T_{B_n^c} =\sum_{b \in B}\pi_B(b)g_{n}(b).\]
Since $\pi(B)<\infty$ we have that 
 $T_{B_{n}^c} \nearrow T_{B^c}<\infty$ a.s. Thus by the monotone convergence theorem we get that  $\mathbb{E}_{b}T_{B_{n}^c} \nearrow  g(b)=\lim_{n \to \infty}g_n(b)$  for all $b \in V$.  Observe that for all $n$ and all $b \in B$ we have that $\mathbb{E}_{b}T_{B_{n}^c} \le g_{n}(b) \le \frac{\pi(B)}{\pi(B_1)}g(b) $, since  $1 \le \frac{\pi_{B_{n}}(b)}{\pi_B(b)} \le \frac{\pi(B)}{\pi(B_1)}<\infty$. Another useful observation is that $\pi_B(b)\mathbb{E}_{b}T_{B_{n}^c} \le \pi_{B_{n}}(b)\mathbb{E}_{b}T_{B_{n}^c}$ for all $b \in B$ and all $n$.  Thus if $\mathbb{E}_{\pi_{B}}T_{B^c} =\infty $, then  by Fatous' lemma \[\liminf_{n \to \infty}\mathbb{E}_{\pi_{B_{n}}}T_{B_n^c} \ge \liminf_{n \to \infty} \sum_{b \in B}\pi_B(b)\mathbb{E}_{b}T_{B_{n}^c} \ge \sum_{b \in B}\pi_B(b)g(b)=\mathbb{E}_{\pi_{B}}T_{B^c} =\infty,\] while if $\mathbb{E}_{\pi_{B}}T_{B^c} <\infty $ then by dominated convergence theorem
\[\lim_{n \to \infty}\mathbb{E}_{\pi_{B_{n}}}T_{B_{n}^c}=\lim_{n \to \infty}\sum_{b \in B}\pi_B(b)g_{n}(b)= \sum_{b \in B}\pi_B(b)g(b)=\mathbb{E}_{\pi_{B}}T_{B^c} <\infty. \]
This concludes the proof of \eqref{e:countable3}. Finally, \eqref{e:countable4}-\eqref{e:countable6} follow easily by combining \eqref{e:countable1}-\eqref{e:countable3} with \eqref{e:2}. 
\end{proof}
\subsection{Proof of Proposition \ref{p:spectral3} }
\label{s:cheeger}
\begin{proof}
The inequality $\chi_{\ell}(u) \le \zeta_{\ell}(u)$ was explained before the statement of Proposition \ref{p:spectral3}. We now prove that $\zeta_{\ell}(u) \le \frac{\ell \log 2}{ \Lambda \left(u 2^{\ell+2}\right)}$.  By\footnote{In \cite{HH} the notation $\|f\|_{p,\nu}$ means $(\sum_x \nu(x)|f(x)|^p)^{1/p}$.} \cite[Lemma 3.7]{HH} (which is a variant of    \cite[Lemma 2.1]{spectral}) we have for all $0 \not \equiv f \in \ell^2(V,\pi) \cap \ell^1(V,\pi) $ that
\begin{equation}
\label{e:spb20111}
\langle( I-P)f ,f \rangle_{\pi}/\|f\|_2^{2} \ge \frac 12 \Lambda(4 \|f\|_1^2/\|f\|_2^2 ). 
\end{equation}
As in the proof of Theorem 2.1 in \cite{spectral}, since $\frac{\mathrm{d}}{\mathrm{d}t}\|P_{t}f\|_2^{2}=-2\langle( I-P)P_{t}f ,P_{t}f \rangle_{\pi}$, Gr\"ownall's lemma and \eqref{e:spb20111} (applied to $f=P_{t}\frac{\mu}{\pi}$) yield that for all $\mu \in \ell^2(\mathcal{P}(V),\pi)$ and all $\delta \in (0,1)$ we have that
\begin{equation}
\label{e:Gwel}
\| \mu P_{t}\|_{2,\pi}^{2} \le \delta \|\mu\|_{2,\pi}^{2} \quad \text{for all} \quad t \ge \int_{4\|\mu\|_{2,\pi}^{-2}}^{4 \delta^{-1} \|\mu\|_{2,\pi}^{-2}}\frac{\mathrm{d}u}{u \Lambda(u)}.  
\end{equation}
Since \[ \int_{4\|\mu\|_{2,\pi}^{-2}}^{2^{\ell+2} \|\mu\|_{2,\pi}^{-2}}\frac{\mathrm{d}u}{u \Lambda(u)} \le \frac{1}{ \Lambda(2^{\ell+2} \|\mu\|_{2,\pi}^{-2})} \int_{4\|\mu\|_{2,\pi}^{-2}}^{2^{\ell+2} \|\mu\|_{2,\pi}^{-2}}\frac{\mathrm{d}u}{u}=\frac{\ell \log 2}{ \Lambda(2^{\ell+2} \|\mu\|_{2,\pi}^{-2})},\] it follows that $\zeta_{\ell}(u) \le \frac{\ell \log 2}{ \Lambda \left(u 2^{\ell+2}\right)}$,  as desired.

It remains to prove that for some absolute constants $c \in (0,1)$ and $\ell \in \mathbb{N}$ we have that $\chi_{\ell}(u) \ge c/\Lambda(u)$. Consider some $D$ such that $\pi(D)<\infty$. By reversibility
\[\| \pi_D P_{t}\|_{2,\pi}^{2} =\sum_x \pi(x)\left(\sum_{b \in D}\pi_D(b) P_{t}(b,x)/\pi(x)\right)\left(\sum_{b \in D}\pi_D(b) P_{t}(x,b)/\pi(b)\right)=\frac{ \mathbb{P}_{\pi_{D}}[X_{2t} \in D]}{\pi(D)}.\]
It follows that $\| \pi_D P_{t}\|_{2,\pi}^{2} /\| \pi_D\|_{2,\pi}^{2}=\mathbb{P}_{\pi_{D}}[X_{2t} \in D] \ge\mathbb{P}_{\pi_{D}}[T_{D^c}>2t]$. Hence, by Theorem \ref{thm:LambdaPucountable} it suffices to show that there exists an absolute constant $c \in (0,1)$ such that for all  $B\in \mathrm{FI}(V,u)$ satisfying that that $\lambda(B) \le c$, there exists some $D \subseteq B$ such that \[\mathbb{P}_{\pi_{D}}[T_{D^{c}}>c/\lambda(B) ] \ge c \]
(the case that $\la(B)>c$ holds trivially with $D=B$, provided $c$ is sufficiently small). This follows from the proof of Theorem \ref{thm:2}. While Theorem \ref{thm:2} is proved in discrete time, the continuous time case of the last claim can easily be derived from the discrete time case by considering a coupling of the discrete time chain with its continuous time version in which they make the same sequence of jumps.
\end{proof}

\bibliographystyle{plain}
\bibliography{hitrelax}

\vspace{2mm}

\end{document}